\documentclass{amsart}

\usepackage[enumitem,theorem,Rn,hyperref,frak,indicator]{paper_diening}

\usepackage{bbm}

\usepackage{microtype}
\usepackage{pgfplots}
  
\usepackage{tikz}
\usetikzlibrary{intersections,calc,arrows.meta,patterns,angles,quotes}

\providecommand{\px}{\ensuremath{{p(\cdot)}}}
\providecommand{\pdx}{\ensuremath{{p'(\cdot)}}}
\providecommand{\phix}{\ensuremath{{\phi(\cdot)}}}
\providecommand{\phidx}{\ensuremath{{\phi^*(\cdot)}}}

\title{New Examples on Lavrentiev Gap Using Fractals}
\dedicatory{Dedicated to Vasili\u{\i} Vasil'evich Zhikov}

\author{Anna Kh.~Balci}

\author{Lars Diening}

\address{Anna Kh.~Balci/Lars Diening, University Bielefeld, Universit\"atsstrasse 25, 33615
  Bielefeld, Germany.}
\email{akhripun@math.uni-bielefeld.de}
\email{lars.diening@uni-bielefeld.de}

\author{Mikhail Surnachev}
\address{Mikhail Surnachev, Keldysh Institute of Applied Mathematics (Russian Academy of Sciences),  Miusskaya sq. 4, 125047
  Moscow,  Russia.}
\email{peitsche@yandex.ru}

\thanks{This research was supported by the Russian Science Foundation
  under grant 19-71-30004 and the German Research
  Foundation (DFG) through the CRC 1283.
}



\begin{document}

\begin{abstract}
  Zhikov showed 1986 with his famous checkerboard example that
  functionals with variable exponents can have a Lavrentiev gap. For
  this example it was crucial that the exponent had a saddle point
  whose value was exactly the dimension.  In 1997 he extended this
  example to the setting of the double phase potential. Again it was
  important that the exponents crosses the dimensional threshold.
  Therefore, it was conjectured that the dimensional threshold plays
  an important role for the Lavrentiev gap. We show that this is not
  the case. Using fractals we present new examples for the Lavrentiev
  gap and non-density of smooth functions. We apply our method to the
  setting of variable exponents, the double phase potential and
  weighted $p$-energy.
\end{abstract}
\maketitle



\providecommand{\opac}{0.2}

\newcommand{\fullpattern}[1]{%
  #1
  
  \repeatpatternthree{#1}
  
  \repeatpatternbig{#1}
}
  
\newcommand\repeatpatternbig[1]{%
  \begin{scope}[shift={(-1,0)},scale=3]
    #1
  \end{scope}
  \begin{scope}[shift={(1,0)},scale=3]
    #1
  \end{scope}
}

\newcommand{\repeatpatternone}[1]{%
  \begin{scope}[shift={(-1/3,0)},scale=1/3]
    #1
  \end{scope}
  \begin{scope}[shift={(1/3,0)},scale=1/3]
    #1
  \end{scope}
}

\newcommand{\repeatpatterntwo}[1]{%
  \begin{scope}[shift={(-1/3,0)},scale=1/3]
    #1
    \repeatpatternone{#1}
  \end{scope}
  \begin{scope}[shift={(1/3,0)},scale=1/3]
    #1
    \repeatpatternone{#1}
  \end{scope}
}

\newcommand{\repeatpatternthree}[1]{%
  \begin{scope}[shift={(-1/3,0)},scale=1/3]
    #1
    \repeatpatterntwo{#1}
  \end{scope}
  \begin{scope}[shift={(1/3,0)},scale=1/3]
    #1
    \repeatpatterntwo{#1}
  \end{scope}
}

\newcommand{\vfullpattern}[1]{%
  #1
  
  \vrepeatpatternthree{#1}
  
  \vrepeatpatternbig{#1}
}

\newcommand\vrepeatpatternbig[1]{%
  \begin{scope}[shift={(0,-1)},scale=3]
    #1
  \end{scope}
  \begin{scope}[shift={(0,1)},scale=3]
    #1
  \end{scope}
}

\newcommand{\vrepeatpatternone}[1]{%
  \begin{scope}[shift={(0,-1/3)},scale=1/3]
    #1
  \end{scope}
  \begin{scope}[shift={(0,1/3)},scale=1/3]
    #1
  \end{scope}
}

\newcommand{\vrepeatpatterntwo}[1]{%
  \begin{scope}[shift={(0,-1/3)},scale=1/3]
    #1
    \vrepeatpatternone{#1}
  \end{scope}
  \begin{scope}[shift={(0,1/3)},scale=1/3]
    #1
    \vrepeatpatternone{#1}
  \end{scope}
}

\newcommand{\vrepeatpatternthree}[1]{%
  \begin{scope}[shift={(0,-1/3)},scale=1/3]
    #1
    \vrepeatpatterntwo{#1}
  \end{scope}
  \begin{scope}[shift={(0,1/3)},scale=1/3]
    #1
    \vrepeatpatterntwo{#1}
  \end{scope}
}

\section{Introduction}
\label{sec:introduction}

The Lavrentiev gap is a phenomenon that may occur in the study of
variational problems. In particular, the minimum of the integral
functional~$\mathcal{G}$ taken over smooth functions may differ from
the one taken over the associated energy space.

The first example for Lavrentiev gap was constructed by Lavrentiev in
\cite{Lav26}. A simpler one was provided by Mani{\'a} in~\cite{Man34}, who considered the functional
\begin{align*}
  \mathcal{G}(w) &:= \int_0^1 \big(x-(w(x))^3\big)^2 \abs{w'(x)}^2\,dx
\end{align*}
subject to the boundary condition~$w(0)=0$ and $w(1)=1$.  Now,
Mani{\'a} showed that there exists~$\tau > 0$ such that
$\mathcal{G}(w) \geq \tau$ for all $w \in C^1([0,1])$ with $w(0)=0$
and $w(1)=1$. However, the function $x^{\frac 13} \in W^{1,1}((0,1))$
has strictly smaller energy, namely $\mathcal{G}(x^{\frac
  13})=0$. This gap between zero and~$\tau$ is the so called
Lavrentiev gap.

In the example of Mani{\'a} the integrand
$f(x,w,\xi) := (x-w^3)^2 \abs{\xi}^2$ depends on~$x$, $w$ and
$\xi$. If the integrand only depends on~$x$ and $\xi$, then the
Lavrentiev gap does not appear in the case of one-dimensional
problems, see~\cite{Lav26}.
The corresponding question for two and higher dimensional problems
with integrands of the form $f(x,\nabla w(x))$ remained open for a
very long time.

\subsection{Zhikov's Famous Checkerboard Example -- Variable Exponents}

In 1986 Zhikov presented his famous two-dimensional checkerboard
example with a Lavrentiev gap, see~\cite{Zhi86}. In particular, he
considered the functional
\begin{align*}
  \mathcal{G}(w) &:= \int\limits_{(0,1)^2} \tfrac{1}{p(x)} \abs{\nabla
                   w}^{p(x)}\,dx 
                   -\int_\Omega b \cdot \nabla w  \, dx
\end{align*}
with the variable exponent (see Figure~\ref{fig:zhikov})
\begin{align*}
  p(x) &:=
         \begin{cases}
           p_2 &\qquad \text{for $x_1 x_2 > 0$},
           \\
           p_1 &\qquad \text{for $x_1 x_2 < 0$},
         \end{cases}
\end{align*}
where $1 < p_1 < 2 < p_2$ and $b \in L^\pdx(\Omega)$ with
$\frac{1}{p'(x)} +\frac{1}{p(x)}=1$. At this~$L^\px(\Omega)$ denotes
the space of Lebesgue space with variable exponent~$p$. The
natural energy space for this functional is $W^{1,\px}_0(\Omega)$, the
space of functions~$w \in W^{1,1}_0(\Omega)$ with
$\nabla w \in L^\px(\Omega)$. See
Subsections~\ref{ssec:gener-orlicz-space}
and~\ref{ssec:variable-exponents} for a precise definition of all
function spaces.  Now, Zhikov constructed a special
vector field~$b \in L^{\pdx}(\Omega)$ such that the infimum
of~$\mathcal{G}$ taken over~$W^{1,\px}_0(\Omega)$ is strictly smaller
than the one taken over the smooth functions~$C^\infty_0(\Omega)$.  If
we denote by~$H^{1,\px}_0(\Omega)$ the closure of~$C^\infty_0(\Omega)$
in $W^{1,\px}(\Omega)$, then we can summarize his result as
\begin{align*}
  \inf \mathcal{G}\big(W^{1,\px}_0(\Omega) \big)
  &<
    \inf \mathcal{G}\big(H^{1,\px}_0(\Omega)\big) =0.
\end{align*}
The important feature of this example is that the exponent $p$ has a
saddle point, where it crosses the dimensional threshold, i.e.
$p_1 < 2 < p_2$. Moreover, the vector field~$b$
satisfies~$\divergence b =0$ in the sense of distributions, so
$\int_\Omega b \cdot \nabla w\,dx=0$ for $w \in C^\infty_0(\Omega)$.
However, for a suitable cut-off
function~$\eta \in C^\infty_0(\Omega)$, we have
$\int_\Omega b \cdot \nabla (\eta u) \,dx = -1$ (see
Proposition~\ref{pro:separating} and~\ref{pro:Su}). Thus, $b$ doesn't
see the gradients of~$C^\infty_0(\Omega)$-functions but it sees the
one of~$\eta u \in W^{1,\phix}_0(\Omega)$.  Therefore, the
vector field~$b$ is also called \emph{separating vector field}. Another useful
feature is that $\abs{b} \cdot \abs{\nabla u}=0$ almost everywhere.

The question of the Lavrentiev gap is closely related to the density
of smooth functions, i.e. if
$W^{1,\px}_0(\Omega) = H^{1,\px}_0(\Omega)$ and
$W^{1,\px}(\Omega) = H^{1,\px}(\Omega)$. Using his checkerboard
exponent Zhikov showed that the function (see Figure~\ref{fig:zhikov})
\begin{align*}
  u(x) &:=
      \begin{cases}
        1 &\qquad \text{on $Q_1 := \set{x_1,x_2 > 0}$},
        \\
        \sin \theta &\qquad \text{on $Q_2 := \set{x_2>0>x_1}$},
        \\
        0 &\qquad \text{on $Q_3 := \set{x_1,x_2 < 0}$},
        \\
        \cos \theta &\qquad \text{on $Q_4 := \set{x_1>0>x_2}$},
      \end{cases}
\end{align*}
defined in polar coordinates $x=(r\cos \theta,r \sin \theta)$ satisfies
$u \in W^{1,\px}(\Omega) \setminus H^{1,\px}(\Omega)$.

The vector field~$b \in L^\pdx(\Omega)$ is in Zhikov's example defined as
$b(x) := \nabla^\perp (u(x^\perp))$ with $x^\perp = (-x_2,x_1)$ and
$\nabla^\perp = (-\partial_2,\partial_1)$, see
\cite[Section~3]{Zhi95}. 
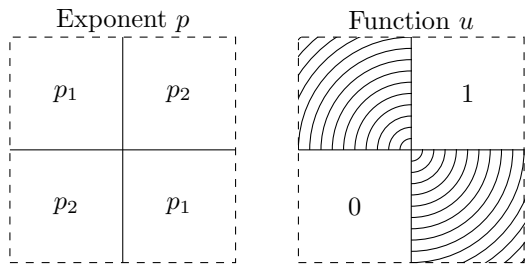
\begin{figure}
  \centering

  \begin{tikzpicture}[scale=1.5]
    \node at (0,1.15) {Exponent~$p$};
    \draw[dashed] (-1,-1) -- (-1,+1) -- (+1,+1) -- (+1,-1) --cycle;
    \draw (-1,0) -- (1,0);
    \draw (0,-1) -- (0,1);
    \node at (0.5,0.5) {$p_2$};
    \node at (-0.5,-0.5) {$p_2$};
    \node at (-0.5,0.5) {$p_1$};
    \node at (0.5,-0.5) {$p_1$};
  \end{tikzpicture}
  \qquad
  \begin{tikzpicture}[scale=1.5]
    \node at (0,1.15) {Function $u$};
    \draw[dashed] (-1,-1) -- (-1,+1) -- (+1,+1) -- (+1,-1) --cycle;
    \draw (-1,0) -- (1,0);
    \draw (0,-1) -- (0,1);

    \node at (0.5,0.5) {$1$};
    \node at (-0.5,-0.5) {$0$};
    \begin{scope}
      \clip (-1,0) rectangle (0,1);
      \draw (0,0) circle (1.6);
      \draw (0,0) circle (1.5);
      \draw (0,0) circle (1.4);
      \draw (0,0) circle (1.3);
      \draw (0,0) circle (1.2);
      \draw (0,0) circle (1.1);
      \draw (0,0) circle (1.0);
      \draw (0,0) circle (0.9);
      \draw (0,0) circle (0.8);
      \draw (0,0) circle (0.7);
      \draw (0,0) circle (0.6);
      \draw (0,0) circle (0.5);
      \draw (0,0) circle (0.4);
      \draw (0,0) circle (0.3);
      \draw (0,0) circle (0.2);
      \draw (0,0) circle (0.1);
    \end{scope}
    \begin{scope}
      \clip (0,-1) rectangle (1,0);
      \draw (0,0) circle (1.6);
      \draw (0,0) circle (1.5);
      \draw (0,0) circle (1.4);
      \draw (0,0) circle (1.3);
      \draw (0,0) circle (1.2);
      \draw (0,0) circle (1.1);
      \draw (0,0) circle (1.0);
      \draw (0,0) circle (0.9);
      \draw (0,0) circle (0.8);
      \draw (0,0) circle (0.7);
      \draw (0,0) circle (0.6);
      \draw (0,0) circle (0.5);
      \draw (0,0) circle (0.4);
      \draw (0,0) circle (0.3);
      \draw (0,0) circle (0.2);
      \draw (0,0) circle (0.1);
    \end{scope}
  \end{tikzpicture}
  
  \caption{Zhikov's checkerboard example; $1 < p_1 < 2 < p_2$. The
    lines indicate smooth transitions.}
  \label{fig:zhikov}
\end{figure}
Now, $u \notin W^{1,p_2}(\Omega)$, since $p_2>2=d$ and
$u$ jumps at the point~$0$. But, since $u$ changes its value in the
area of the exponent~$p_1 <2$, we still get
$u \in W^{1,p(\cdot)}(\Omega)$.

However, $u$ cannot be approximated by smooth functions~$u_n$. Indeed,
it follows from $u_n\in W^{1,p_2}(Q_1) \cap W^{1,p_2}(Q_3)$ and the
continuity of~$u_n$ at~$0$ that the $u_n$ are uniformly H\"older
continuous on $\overline{Q_1} \cup \overline{Q_3}$ with exponent
$\alpha := 1-\frac{2}{p_2}>0$ uniformly in~$n$, but the limit~$u$ is
not even continuous.

It is possible to generalize this example to higher dimensions, i.e.
$\Omega = (-1,1)^d$ with~$d\ge 2$. Again the variable
exponent~$p$ has a saddle point at zero. It takes the value $p_2 > d$ on
the double
cone~$\Omega \cap \set{x=(x',x_d)\,:\,\abs{x_d} \geq \abs{x'}}$ and
$p_1 < d$ on its complement. The crucial point is again
that~$W^{1,p_2}(\Omega) \embedding C^{0,\alpha}(\overline{\Omega})$
for $\alpha = 1-\frac{d}{p_2}$. So the exponent has a saddle point,
where it just crosses the dimension~$d$.

Up to now no other examples of
$H^{1,\px}(\Omega) \neq W^{1,\px}(\Omega)$ with variable exponents
were known. This led people to the question if the dimension plays a
critical role, i.e. that it is important that at the saddle point the
variable exponent~$p$ crosses the threshold~$d$.  This question has been
repeatedly raised by Zhikov and also by
H\"ast\"o~\cite{DieHasNek04}.

The saddle point setup is the simplest geometry for the Lavrentiev gap
to appear. It has been shown by Zhikov~\cite{Zhi86} that if~$\px$
takes only two values separated by a smooth surface,
then~$H^{1,\px}(\Omega)=W^{1,\px}(\Omega)$. Even more, if we take a piecewise constant
exponent which takes three constant values in three sectors of the
plane separated by rays emanating from the origin, then
also~$H^{1,\px}=W^{1,\px}$. This is a special case of the montonicity
condition on cones by Edmunds and Rakosnik~\cite{EdmRak92} that
ensures $H^{1,\px}(\Omega)=W^{1,\px}(\Omega)$.

Another situation for~$H^{1,\px}(\Omega)=W^{1,\px}(\Omega)$ is,
when~$p$ has certain regularity.  In 1995 Zhikov~\cite{Zhi95} found
the celebrated local $\log$-Hölder condition
\begin{align*}
  \abs{p(x)-p(y)}\le \frac{c}{\log (e+\frac{1}{\abs{x-y}})}.
\end{align*}
This condition allows to use mollification to prove density of smooth
functions~\cite{Zhi95}, \cite{Die04}, \cite[Section
4.6]{DieHHR11}. The $\log$-H\"older continuity is also important for
many other properties like boundedness of the maximal operator, the
Riesz potential, singular integral operators and sharp Sobolev
embeddings. For more details we refer to the
books~\cite{DieHHR11,CruFio13,KokMesRafSam16}. For the density of
smooth functions it is possible to weaken the modulus of continuity
slightly by an extra double log factor, see~\cite{Zhi04}.

In Zhikov's original example the exponent~$p(\cdot)$ jumped at the
saddle point. However, it is possible to modify the exponent to a
uniformly continuous one (not $\log$-H\"older continuous). Again the
exponent has a saddle point, where it crosses the dimensional
threshold. Such examples have been obtained independently by
Zhikov~\cite{Zhi95} and H\"ast\"o~\cite{Has05}.

Zhikov also showed with his counter example that the notion of
$p(\cdot)$-harmonic functions becomes ambiguous. In particular, minimizers
of
\begin{align*}
  \mathcal{F}(w) &:= \int_\Omega \tfrac 1{p(x)} \abs{\nabla w}^{p(x)}\,dx
\end{align*}
for given nice boundary may differ depending if we minimize in
$H^{1,\px}(\Omega)$ or in $W^{1,\px}(\Omega)$.

One of the goals of this paper is to provide new examples of variable
exponents, such that the Lavrentiev gap occurs, but which do not need
to cross the dimensional threshold, see
Subsection~\ref{ssec:variable-exponents}. We also show the non-density
of smooth functions, i.e. $H^{1,\px}(\Omega) \neq W^{1,\px}(\Omega)$
and the ambiguity of $\px$-harmonicity.

\subsection{Double Phase Potential}

The famous checkerboard example became the guiding principle for other
models. In 1995 Zhikov~\cite[Example~3.1]{Zhi95} considered the double phase potential
\begin{align*}
  \mathcal{F}(w) &:= \int_\Omega \tfrac 1 p \abs{\nabla w}^p + \tfrac 1q
                   a(x) \abs{\nabla w}^q\,dx,
\end{align*}
where $1 < p < q < \infty$ and~$a \in C^{0,\alpha}(\overline{\Omega})$
is a non-negative weight. He constructed with a similar checkerboard
setup a weight~$a \in C^{0,\alpha}(\overline{\Omega})$ with
$\alpha=1$, $\Omega= (-1,1)^2$ and $p < 2 < 2+\alpha =3 < q$ such that
the Lavrentiev gap occurs. On the quadrants~$Q_1$ and~$Q_3$ he
chose~$a(x) = \frac{x_1x_2}{\abs{x}}$ and on the quadrants~$Q_2$
and~$Q_4$ he chose~$a(x)=0$. The exponents take the same values as in
Figure 1 with~$p_1=p$ and~$p_2=q$.

Again he showed that there exists a
functional~$\mathcal{G}(w) = \mathcal{F}(w) + \int_\Omega b \cdot
\nabla w\,dx$ such that
\begin{align*}
  \inf \mathcal{G}\big(W^{1,p}_0(\Omega) \big)
  &<
    \inf \mathcal{G}\big(W^{1,q}_0(\Omega)\big).
\end{align*}
This example was generalized by Esposito, Leonetti and Mingione
in~\cite{EspLeoMin04} to the case of higher dimensions and less
regular weights, i.e. $\alpha \in (0,1]$. In particular
for~$\Omega=(-1,1)^d$, they constructed a
weight~$a \in C^{0,\alpha}(\overline{\Omega})$ and exponents
$1 < p < d < d+ \alpha < q$ such that the Lavrentiev gap occurs.  For
this they changed $a$ 
to~$a(x)=\abs{x}^\alpha \max
\set{\frac{{x_d}^2-\abs{\bar{x}}^2}{\abs{x}^2},0}$ with
$x=(\bar{x},x_d)$.  In both examples by Zhikov and
Esposito-Leonetti-Mingione the two exponents~$p$ and~$q$ cross the
dimensional threshold~$d$. So again, there was the question, if this
threshold is important for the Lavrentiev gap.

This phenomenon for the double phase potential can also be seen as a
lack of higher regularity, see Marcellini~\cite{Mar89} for the first
example in this direction. In fact, local minimizers of~$\mathcal{F}$
need not be~$W^{1,q}$-functions unless~$a$, $p$ and~$q$ satisfy
certain assumptions. In fact, if $\frac{q}{p} \leq 1+\frac \alpha{d}$
and $a \in C^{0,\alpha}$, then minimizers of~$\mathcal{F}$ are
automatically in~$W^{1,q}$, see \cite{ColMin15}. Moreover, bounded
minimizers of~$\mathcal{F}$ are automatically~$W^{1,q}$
if~$a \in C^{0,\alpha}$ and~$q \leq p+\alpha$,
see~\cite{BarColMin18}. If the minimizer is from~$C^{0,\gamma}$, then
the requirement can be relaxed to $q \leq p + \frac{\alpha}{1-\gamma}$
\cite[Theorem~1.4]{BarColMin18}. The example of
Esposito-Leonetti-Mingione shows that in some sense these estimates
are sharp. However, they are sharp only for $p=d-\epsilon$
with~$\epsilon>0$ small.  In this paper we will provide new examples
for the Lavrentiev gap that get rid of this
condition~$p = d-\epsilon$. We will present new examples that show
that the conditions~$q \leq p+\alpha$ and
$q \leq p + \frac{\alpha}{1-\gamma}$ are sharp for a much wider range
of~$p$ and~$q$, see Subsection~\ref{ssec:double-phase-potent}. In
particular, we present examples without the dimensional threshold.

The question of Lavrentiev gap can also be viewed from the point of
function spaces. In fact, the energy~$\mathcal{F}$ defines a
generalized Sobolev-Orlicz space~$W^{1,\phix}(\Omega)$ and its
counterpart~$W^{1,\phix}_0(\Omega)$ with zero boundary values, see
Subsection~\ref{ssec:gener-orlicz-space} for the precise definition of
the spaces. Then the above Lavrentiev gap can be also written as
\begin{align*}
  \inf \mathcal{G}\big(W^{1,\phix}_0(\Omega) \big)
  &<
    \inf \mathcal{G}\big(H^{1,\phix}_0(\Omega)\big),
\end{align*}
where $H^{1,\phix}_0(\Omega)$ is the closure of~$C^\infty_0(\Omega)$
functions in~$W^{1,\phix}(\Omega)$. Hence, the question of the
Lavrentiev cap is closely related to the density of smooth functions,
i.e. if $H^{1,\phix}_0(\Omega) = W^{1,\phix}_0(\Omega)$ and
$H^{1,\phix}(\Omega) = W^{1,\phix}(\Omega)$.

We present fractal examples without the dimensional threshold that
support the Lavrentiev gap, the non-density of smooth functions and
the ambiguity of the related harmonicity.

\subsection{\texorpdfstring{Weighted $p$-Energy}{Weighted p-Energy}}
\label{ssec:zhikov-weight-sobol-spac}

Zhikov also considered another example, namely the one of weighted
Sobolev spaces. In particular, he considered the energy
\begin{align*}
  \mathcal{F}(w) := \int_\Omega \tfrac 1p a(x) \abs{\nabla w}^p\,dx
  =\int_\Omega \tfrac 1p (\omega(x) \abs{\nabla w})^p\,dx.
\end{align*}
Again, he used a checkerboard setup to construct weights~$a$,
resp.~$\omega$ that provide for~$p=2$ a Lavrentiev gap and non-density
of smooth functions, see \cite[Example~3.3]{Zhi95}.  His weight is
unbounded but it is bounded from above and below by two Muckenhoupt
weights from~$A_2$.  In \cite[Section 5.3]{Zhi98} he presented another
more complicated example with a bounded weight, see
Remark~\ref{rem:zhikov-barrier}.

Again, we present fractal examples without the dimensional threshold
that support the Lavrentiev gap, the non-density of smooth functions
and the ambiguity of the related harmonicity.

If~$a$ itself is a Muckenhoupt weight, then it is well known that
smooth functions are dense, so
$W^{1,\phix}(\Omega) = H^{1,\phix}(\Omega)$. For other results on the
density in the context of weighted Sobolev spaces with even variable
exponents, we refer to~\cite{Sur14,ZhiSur16}.

\subsection{Structure of the Article}

The structure of the article is as follows. In
Section~\ref{sec:constr-fract} we will use fractals to construct the
functions~$u$ and~$b$ that we need later in our applications. We start
with a modified version of the checker board example by Zhikov, which
works in all dimensions. Then we introduce the necessary fractals of
Cantor type to construct function~$u$ and the vector field~$b$ without the problem of
the dimensional threshold.

In Section~\ref{sec:consequences} we show how~$u$ and~$b$ can be used
to deduce the Lavrentiev gap, the non-density of smooth functions and
ambiguity of the related harmonicity. In this section we also introduce the
necessary function spaces.

In Section~\ref{sec:applications} we apply out technique to the model
of variable exponents, the double phase potential and weighted
$p$-energy. From the point of applications these are the main results
of our paper.


\section{Construction of Fractal Examples}
\label{sec:constr-fract}

In this section we will use fractals to construct the functions~$u$
and~$b$, which are necessary to study later in
Section~\ref{sec:consequences} the Lavrentiev gap and the other
phenomena. The construction of these functions is independent of the
models that we consider in
Section~\ref{sec:applications}. 


Let us clarify our notation. By $B^m_r(x)$ we denote the ball
of~$\setR^m$ with radius~$r$ and center~$x$. We denote
by~$\indicator_A$ the indicator function of the
set~$A$. By~$L^p(\Omega)$ and $W^{1,p}(\Omega)$ we denote the usual
Lebesgue and Sobolev spaces. Moreover, let~$W^{1,p}_0(\Omega)$ be the
Sobolev space with zero boundary values. By~$L^1_{\loc}(\Omega)$ we
denote the space of locally integrable functions (integrable on
compact subsets) with $W^{1,1}_{\loc}(\Omega)$ defined analogously. We
use~$c>0$ for a generic constants whose value may change from line to
line but does not depend on critical parameters. We also abbreviate
$f \lesssim g$ for $f \leq c\, g$.

\subsection{One Building Block}
\label{ssec:one-building-block}

We begin with a multidimensional, revised version of the Zhikov
example. We will use it later as the building block for fractal
examples.

\begin{definition}[Zhikov's example; revised]
  \label{def:zhikov}
  For~$d \geq 2$ define $u_d$, $A_d$ and $b_d$ on~$\Rd$ by
  \begin{align*}
    u_d &:= \sgn(x_d)\, \theta\bigg( \frac{\abs{x_d}}{\abs{\bar{x}}} \bigg), 
    \\
    A_d(x) &:= \theta\bigg( \frac{\abs{\bar{x}}}{\abs{x_d}} \bigg)
            \frac{1}{\sigma_{d-1}} \abs{\bar{x}}^{1-d}
              \begin{pmatrix}
                0 & -\bar{x}\\
                \bar{x}^T & 0
              \end{pmatrix}, 
    \\
    b_d &:= \divergence A_d,
  \end{align*}
  where $\sigma_{d-1}$ is the surface area of the $d-1$-dimensional
  sphere and ~$\theta \in C^\infty_0((0,\infty))$ is such that
  $\indicator_{(\frac 12,\infty)} \leq \theta \leq \indicator_{(\frac
    14, \infty)}$,  $\norm{\theta'}_\infty \le 6$.
\end{definition}

The matrix divergence is taken rowwise, i.e for matrix~$A=\set{A_{ij}}$ we define~$(\divergence A)_i=\partial_j A_{ij}$. 

In Figure~\ref{fig:revised-checkerboard} it is shown how our revised
version of Zhikov's checkerboard example looks for~$d=2$. The picture
shows the function~$u_2$, the $(2,1)$-component of~$A_2$ and a
possible exponent~$p$. The picture should be compared to the one of
Figure~\ref{fig:zhikov}. There are two main differences. First, our
version is rotated by~$45^\circ$ counterclockwise. Second, there is an
additional area, where~$u_2=0$. This fact will be very useful
later. Note that on the shaded region of~$u_2$, resp. $A_2$, we have
$\abs{x_1} \eqsim \abs{x_2} \eqsim \abs{x}$, which allows us later a
freedom in the choice of variable on the shaded region.
\begin{figure}[!ht]
  \centering
  
  \begin{tikzpicture}[scale=1.5]
    \node at (0,1.15) {Function~$u_2$};
    \draw[dashed] (-1,-1) -- (-1,+1) -- (+1,+1) -- (+1,-1) --cycle;
    \node at (0,.7) {\scalebox{0.8}{$\frac 12$}};
    \node at (0,-.7) {\scalebox{0.8}{$-\frac 12$}};
    \draw (-1,-1/2) -- (1,1/2);
    \draw (-1,1/2)-- (1,-1/2);
    \node at (+.6,0) {\scalebox{0.8}{$0$}};
    \node at (-.6,0) {\scalebox{0.8}{$0$}};
    \filldraw[pattern=vertical lines] (-1,-1/2) -- (0,0) -- (-1,-1/4);
    \filldraw[pattern=vertical lines] (-1,+1/2) -- (0,0) -- (-1,+1/4);
    \filldraw[pattern=vertical lines] (+1,-1/2) -- (0,0) -- (+1,-1/4);
    \filldraw[pattern=vertical lines] (+1,+1/2) -- (0,0) -- (+1,+1/4);
  \end{tikzpicture}
  \quad
  \begin{tikzpicture}[scale=1.5]
    \node at (0,1.15) {$(2,1)$-component of~$A_2$};
    \draw[dashed] (-1,-1) -- (-1,+1) -- (+1,+1) -- (+1,-1) --cycle;
    
    \node at (.6,0) {\scalebox{0.8}{$\frac 12$}};
    \node at (-.55,0) {\scalebox{0.8}{$-\frac 12$}}; 
    \draw (-1/2,-1)--(+1/2,1);
    
    \draw(+1/2,-1)-- (-1/2,+1) ;
    \node at (0,.7) {\scalebox{0.8}{$0$}};
    \node at (0,-.7) {\scalebox{0.8}{$0$}};
    
    \filldraw[pattern=horizontal lines] (-1/2,-1) -- (0,0) -- (-1/4,-1);
    \filldraw[pattern=horizontal lines] (+1/2,-1) -- (0,0) -- (+1/4,-1);
    \filldraw[pattern=horizontal lines] (-1/2,+1) -- (0,0) -- (-1/4,+1);
    \filldraw[pattern=horizontal lines]
    (+1/2,+1) -- (0,0) -- (+1/4,+1);

  \end{tikzpicture}
  \quad
  \begin{tikzpicture}[scale=1.5]
    \node at (0,1.15) {Exponent~$p$};
    \draw[dashed] (-1,-1) -- (-1,+1) -- (+1,+1) -- (+1,-1) --cycle;
    \node at (0,.7) {\scalebox{0.8}{$2+\epsilon$}};
    \node at (0,-.7) {\scalebox{0.8}{$2+\epsilon$}};
    \node at (.6,0) {\scalebox{0.8}{$2-\epsilon$}};
    \node at (-.7,0) {\scalebox{0.8}{$2-\epsilon$}};
    \draw (0,0) -- ( 1, 1)-- (-1, 1)--cycle;
    \draw (0,0) -- ( 1,-1)-- (-1,-1)--cycle;
    \draw (0,0) -- ( 1, 1)-- ( 1,-1)--cycle;
    \draw (0,0) -- (-1, 1)-- (-1,-1)--cycle;
    \draw (0,0) -- (+1,+1);
    \draw (0,0) -- (+1,-1);
    \draw (0,0) -- (-1,+1);
    \draw (0,0) -- (-1,-1);
  \end{tikzpicture}
  \caption{Revised Checkerboard Example}
  \label{fig:revised-checkerboard}
\end{figure}
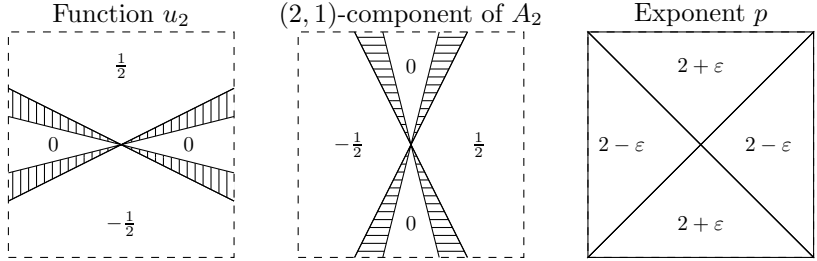
Another difference to the example of Zhikov is the improved regularity
away from the singularity at zero.
\begin{proposition}
  \label{pro:prop-uAb-d}
  There holds
  \begin{enumerate}
  \item
    $u_d \in L^\infty(\Rd) \cap W^{1,1}_{\loc}(\Rd) \cap C^\infty(\Rd
    \setminus \set{0})$,
  \item $A_d \in W^{1,1}_{\loc}(\Rd) \cap C^\infty(\Rd \setminus \set{0})$,
  \item
    $b_d \in L^1_{\loc}(\Rd) \cap C^\infty(\Rd \setminus \set{0})$.
  \end{enumerate}
  Moreover, the following estimates hold
  \begin{align*}
    \begin{alignedat}{2}
      \abs{\nabla u_d} &\lesssim \abs{x_d}^{-1} \indicator_{ \set{ 2
          \abs{x_d} \leq \abs{\bar{x}} \leq 4 \abs{x_d}}} &&\eqsim
      \abs{\bar{x}}^{-1} \indicator_{ \set{ 2 \abs{x_d} \leq
          \abs{\bar{x}} \leq 4 \abs{x_d}}}
      \\
      \abs{b_d} &\eqsim \abs{x_d}^{1-d}
      \indicator_{ \set{ 2 \abs{\bar{x}} \leq \abs{x_d} \leq 4
          \abs{\bar{x}}}} .
     \end{alignedat}
  \end{align*}
  In particular, $\abs{\nabla u_d}\cdot \abs{b_d} = 0$.
\end{proposition}
\begin{proof}
  It is easy to see that
  $u_d,A_d,b_d \in C^\infty(\Rd \setminus \set{0})$ and
  $u_d \in L^\infty(\Rd)$. Moreover, $u_d$ and $A_d$ have the ACL
  property (absolutely continuous on almost every line parallel to the
  axis).  Now, the estimates for $\abs{\nabla u_d}$, $\abs{b_d}$ 
   are also straight forward. They imply immediately
  that $\nabla u_d,  b_d \in L^1_{\loc}(\Rd)$. This proves
  the claim.
\end{proof}
Note that
\begin{align}
  \label{eq:bd-divfree}
  \divergence b_d = \divergence (\divergence A_d)=0 \qquad
  \text{on~$\setR^d \setminus \set{0}$,}
\end{align}
since~$A_d$ is skew-symmetric. Moreover, for all~$\xi \in
C^\infty_0(\Rd)$ we have
\begin{align*}
  \int_{\Rd} b_d \cdot \nabla \xi \,dx
  &=
    \int_{\Rd} \divergence A_d \cdot \nabla \xi \,dx = -
    \int_{\Rd} A_d  : \nabla^2 \xi \,dx  =0,
\end{align*}
again, since~$A_d$ is skew-symmetric. Thus,
\begin{align}
  \label{eq:divbd}
  \divergence b_d &=\divergence (\divergence A_d) = 0 \qquad \text{in $(C^\infty_0(\Rd))'$, i.e. in
                    the distributional sense}.
\end{align}
This is a crucial property of~$b_d$, since it implies that~$b_d$ is
orthogonal to the gradient of smooth functions. This will allow us
later to separate~$u_d$ from the smooth functions. Therefore, we
call~$b_d$ also \emph{separating vector field}.

It follows from~\eqref{eq:divbd} and the regularity of~$b_d$ that
\begin{align}
  \label{eq:divbd2}
  \divergence b_d &=0 \qquad \text{in the classical sense on~$\Rd
  \setminus \set{0}$.}
\end{align}

Another important feature of~$u_d$ and~$b_d$ is the following
proposition on the boundary integral that we would obtain it if we were
allowed to use partial integration
on~$\int_\Omega b_d \cdot \nabla u_d\,dx$ (we are not,
since~$b_d \notin W^{1,1}(\Omega)$).
\begin{proposition}
  \label{pro:match-int-bu}
  Let $\Omega= (-1,1)^d$ with~$d \geq 2$. Then
  \begin{align*}
    \int_{\partial \Omega} (b_d \cdot \nu) u_d\,dS = 1.
  \end{align*}
\end{proposition}
For a better understanding we show in Figure~\ref{fig:boundary} the
values for~$b_2 u_2$ that we need in
Proposition~\ref{pro:match-int-bu}. 
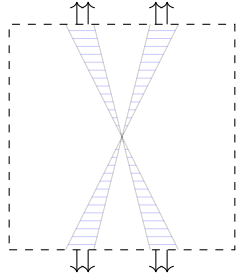
\begin{figure}[!ht]
  \centering
  \begin{tikzpicture}[scale=1.5]
    \draw[dashed] (-1,-1) -- (-1,+1) -- (+1,+1) -- (+1,-1) --cycle;

    \draw[->] (-0.4,1) -- (-0.4,1.2);
    \draw[->] (-0.3,1) -- (-0.3,1.2);
    \draw[->] (0.4,1) -- (0.4,1.2);
    \draw[->] (0.3,1) -- (0.3,1.2);
    \draw[->] (-0.4,-1) -- (-0.4,-1.2);
    \draw[->] (-0.3,-1) -- (-0.3,-1.2);
    \draw[->] (0.4,-1) -- (0.4,-1.2);
    \draw[->] (0.3,-1) -- (0.3,-1.2);

    \filldraw[pattern=horizontal lines, pattern color=blue,opacity=0.2] (-1/2,-1) -- (0,0) -- (-1/4,-1);
    \filldraw[pattern=horizontal lines, pattern color=blue,opacity=0.2] (+1/2,-1) -- (0,0) -- (+1/4,-1);
    \filldraw[pattern=horizontal lines, pattern color=blue,opacity=0.2] (-1/2,+1) -- (0,0) -- (-1/4,+1);
    \filldraw[pattern=horizontal lines, pattern color=blue,opacity=0.2]
    (+1/2,+1) -- (0,0) -- (+1/4,+1);

  \end{tikzpicture} 
  \caption{Boundary values of~$b_2 u_2$ used for $\int_{\partial
      \Omega} (b_2 \cdot \nu)u_2 \,dx =1$.}
  \label{fig:boundary}
\end{figure}
Before we get to the proof we need the following lemma.
\begin{lemma}
  \label{lem:bd-int}
  There holds
  $(\bar{x} \mapsto b_d(\bar{x},1)) \in
  C^\infty_0((-1,1)^{d-1})$ and
  \begin{align*}
    \int_{\setR^{d-1}} b_{d} (\bar{x},1) \cdot e_d\,d\bar{x} &= 1.
  \end{align*}
\end{lemma}
\begin{proof}
  Let $B := B^{d-1}_{\sqrt{d}}(0) \subset \setR^{d-1}$. Then
  $\support (b_d(\cdot,1)) \subset B$ and
  \begin{align*}
    I := \int_{\setR^{d-1}} b_{d} (\bar{x},1) \cdot e_d\,d\bar{x}
    &=\int_B  (\divergence A_d)(\bar{x},1) \cdot
      e_d\,d\bar{x}.
  \end{align*}
  Let use define $g \,:\, \setR^{d-1} \to \setR^{d-1}$ as
  \begin{align*}
    g(\bar{x}) &:= \theta
                 \bigg(\frac{\abs{\bar{x}}}{\abs{x_d}} \bigg)
                 \frac{1}{\sigma_{d-1}} \frac{\bar{x}}{\abs{\bar{x}}^{d-1}}
  \end{align*}
  Then by the definition of~$A_d$ we get
  \begin{align*}
    (\divergence A_d)(\bar{x},1) \cdot
    e_d &= \divergence_{\bar{x}} g.
  \end{align*}
  Thus, by the theorem of Gau{\ss}
  \begin{align*}
    I
    &=
      \int_{B} \divergence_{\bar{x}} g \,d\bar{x} =
      \int_{\partial B} g \cdot \nu\,dS =
      \int_{\partial B}                  \frac{1}{\sigma_{d-1}} \abs{\bar{x}}^{2-d}
      \,dS = 1
  \end{align*}
  using that~$\sigma_{d-1}$ is the surface area of the
  $d-1$-dimensional sphere. This proves the claim.
\end{proof}
We can now prove Proposition~\ref{pro:match-int-bu}.
\begin{proof}[Proof of Proposition~\ref{pro:match-int-bu}]
  Note that~$b_d=0$ on $\partial \Omega$ except for the
  sets~$\set{x_d=\pm 1} \cap \partial \Omega$. On these sets $u_d$ takes the
  values~$\pm \frac 12$ and $\nu = \pm e_d$. Moreover, $b_d\cdot e_d$ is even
  with respect to~$x_d$. Thus,
  \begin{align*}
    \int_{\partial \Omega} (b \cdot \nu) u_d\,dS = 
    \int_{(-1,1)^{d-1}} b(\bar{x},1) \cdot e_d\,d\bar{x} = 1
  \end{align*}
  using Lemma~\ref{lem:bd-int}.
\end{proof}

\subsection{Cantor Sets}
\label{ssec:cantor-sets}

In the Zhikov's example the contact set~$\frS$ consists  just of
one point, the origin, which has dimension zero. For our new examples
we want to use contact sets of higher, fractal dimension. For this
reason we start with the definition of a few fractal Cantor sets that
we need later.

We begin with the one dimensional generalized Cantor
set~$\frC_\lambda$ with $\lambda \in (0,\frac 12)$, which is also
known as the (1-2$\lambda$)-middle Cantor set. We start with the
interval~$\frC_{\lambda,0} := (-\frac 12,\frac 12)$. Then we
define~$\frC_{\lambda,k+1}$ inductively by removing the
middle~$1-2\lambda$ parts from~$\frC_{\lambda,k}$. In particular, we
define~$\frC_\lambda := \cap_{k \geq 1} \frC_{\lambda,k}$.  The
corresponding Cantor measure~$\mu_\lambda$ (also Cantor distribution)
is then defined as the weak limit of the
measures~$\mu_{\lambda,k}:=(2\lambda)^{-k}
\indicator_{\frC_{\lambda,k}}\,dx$. The factor $(2 \lambda)^{-k}$ is
chosen such
that~$\mu_{\lambda,k}([-\frac 12,\frac 12])=1= \mu([-\frac 12,\frac
12])$.  Thus, $\mu_\lambda(\setR)=1$ and
$\support \mu_\lambda = \frC_\lambda$.  The fractal dimension
of~$\frC_\lambda$ is
$\dim(\frC_\lambda) = \log(2)/ \log(1/\lambda) \in (0,1)$,
i.e. $\lambda = 2^{-\frD}$.

We will also need the~$m$-dimensional Cantor sets~$\frC_\lambda^m$ and
its distribution~$\mu^m_\lambda$, which are just the Cartesian
products of $\frC_\lambda$ and $\mu_\lambda$. Its fractal dimension
is~$\dim \frC_\lambda^m = m \dim \frC_\lambda = m \log(2)/ \log(1/
\lambda) \in (0,m)$. Note that
$\frC_\lambda^m = \cap_{k \geq 1} \frC_{\lambda,k}^m$.

\begin{figure}[!ht]
  \centering
  \begin{tikzpicture}[scale=8]
    \draw[ultra thick] (0,0) -- (1,0);
    \draw[ultra thick] (0,-.05) -- (1/3,-0.05);
    \draw[ultra thick] (2/3,-.05) -- (1,-0.05);
    \draw[ultra thick] (0,-.1) -- (1/9,-0.1);
    \draw[ultra thick] (1/9+1/9,-.1) -- (1/3,-0.1);
    \draw[ultra thick] (2/3,-.1) -- (2/3+1/9,-0.1);
    \draw[ultra thick] (2/3+2/9,-.1) -- (1,-0.1);
    \node[anchor=east] at (0,0) {$\frC_{\lambda,0}$};
    \node[anchor=east] at (0,-0.05) {$\frC_{\lambda,1}$};
    \node[anchor=east] at (0,-0.1) {$\frC_{\lambda,2}$};
  \end{tikzpicture}
  
  \caption{Construction of Generalized Cantor Set~$\frC_{\frac 13}$.}
  \label{fig:cantor-lambda}
\end{figure}
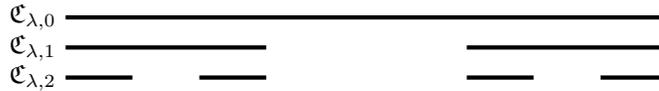
In the construction of our fractal examples we need a
smooth approximation of the indicator function
$\indicator_{\set{ d(\bar{x},\frC^{m}_\lambda) \leq 3
    \abs{\hat{x}}}}$, where
$x= (\bar{x},\hat{x}) \in \setR^m \times \setR^{d-m}$. This is the
purpose of the following lemma.

\begin{lemma}
  \label{lem:smooth-indicator}
  Let~$\lambda \in (0,\frac 12)$, $1 \leq m \leq d$ and
  let~$\frS := \frC^m_\lambda \times \setR^{d-m}$. We use the
  notation $x=(\bar{x},\hat{x}) \in \setR^m \times \setR^{d-m}$. Let
  $\frac 14 \leq \tau_1 < \tau_2 \leq 4$ and~$\tau_2-\tau_1 \geq \frac
  14$. Then there exists~$\rho \in C^\infty(\setR^d \setminus \frS)$
  such that
  \begin{enumerate}
  \item \label{itm:rho1}
    $\indicator_{\set{d(\bar{x},\frC^m_\lambda) \leq \tau_1 \abs{\hat{x}}}}
    \leq \rho \leq \indicator_{\set{ d(\bar{x},\frC^m_\lambda) \leq
        \tau_2 \abs{\hat{x}}}}$.
  \item \label{itm:rho2}
    $\abs{\nabla \rho(\bar{x},\hat{x})} \lesssim \abs{\hat{x}}^{-1}
    \indicator_{\set{\tau_1 \abs{\hat{x}} \leq d(\bar{x},\frC^m_\lambda) \leq
        \tau_2 \abs{\hat{x}}}}$.
  \end{enumerate}
  In particular, $\rho=1$
  on~$\set{d(\bar{x},\frC^m_\lambda) \leq \tau_1 \abs{\hat{x}}}$ and
  $\rho=0$ on $\set{\tau_2 \abs{\hat{x}} \leq d(\bar{x},\frC^m_\lambda)}$.
\end{lemma}
\begin{proof}
  Let $\tau := \frac{\tau_2-\tau_1}{2}$. The function
  $\indicator_{\set{\tau \abs{\hat{x}} \geq d(\bar{x},\frC^m_\lambda) }}$
  would satisfy~\ref{itm:rho1} but not the smoothness
  requirement. Therefore, we need to mollify this function depending
  on the~$\abs{\hat{x}}$-value. For this let $\set{\psi_t}$ denote a
  standard mollifier, i.e. $\support(\psi_1) = \overline{B^d_1(0)}$,
  $\psi_1\geq 0$, $\int \psi_1\,dx=1$, $\psi_1 \in C_0^\infty(B^d_1(0))$ and  
  $\psi_t(x) = t^{-d}\psi(x/t)$.
  \begin{align*}
    \rho(x) &:= \big(\indicator_{\set{\tau
              \abs{\hat{x}} \geq d(\bar{x},\frC^m_\lambda) }} *
              \psi_{\frac{\abs{\hat{x}}}{100}}\big)(x). 
  \end{align*}
  The factor~$\frac{1}{100}$ in the scaling of the mollifier is chosen
  so small such that the smeared version of the jump set
  $\set{\tau\abs{\hat{x}} = d(\bar{x},\frC^m_\lambda)}$ stays inside
  $\set{\tau_1 \abs{\hat{x}} \leq d(\bar{x},\frC^m_\lambda) \leq \tau_2
    \abs{\hat{x}}}$. This proves~\ref{itm:rho1}. Now, the standard estimate
  implies~\ref{itm:rho2}.  It is obvious
  that~$\rho \in C^\infty(\setR^d \setminus (\setR^m \times
  \set{0}^{d-m}))$. Moreover, since $\rho=0$ on $\set{\tau_2 \abs{\hat{x}} < d(\bar{x},\frC^m_\lambda)}$ it follows that~$u$ is also~$C^\infty$ at $(\setR^{m} \times \{0\}^{d-m}) \setminus \frS$. This proves
  that~$u \in C^\infty(\setR^d \setminus \frS)$.
\end{proof}
The following two lemmas provides further technical estimates that are used
later to determine the integrability of our fractal examples. 

\begin{lemma}
  \label{lem:cantor-estimates}
  Let~$\lambda \in (0,\frac 12)$, $1\leq m \leq d$ and
  $\frD := \dim(\frC_\lambda^m) = -m \log(2)/\log(\lambda)$. We
  use the notation $x=(\bar{x},\hat{x}) \in \setR^m \times \setR^{d-m}$.
  Then we have the following properties: 
  \begin{enumerate}
    \setlength{\itemsep}{1ex}%
  \item \label{itm:cantor-estimates1} For  every
    ball~$B^m_r(\bar{x})$ there
    holds~$\mu^m_\lambda (B^m_r(\bar{x})) \lesssim
    \indicator_{\set{d(\bar{x},\frC^m_\lambda)\le r}} r^{\frD}$.

  \item \label{itm:cantor-estimates2} For all $r>0$ there holds
    $
    \mathcal{L}^m\big(                                \set{
      \bar{x}:   
      d(\bar{x},\frC^m_\lambda) \leq  r } \big)
    \lesssim r^{m-\frD}$.
  \item \label{itm:cantor-estimates3}
    For all $\tau \in (0,4]$ there
    holds
    \begin{align*}
      \bigabs{\big((\mu_\lambda^m \times \delta_0^{d-m}) *
      \indicator_{\set{\abs{\bar{x}}\leq \tau\abs{\hat{x}}}}\big)(x)} 
      \lesssim
      \indicator_{ \set{ 
      d(\bar{x},\frC^m_\lambda) \leq\tau \abs{\hat{x}}}}(x)
      \,\abs{\hat{x}}^\frD.
    \end{align*}
  \end{enumerate}
\end{lemma}
\begin{proof}
  Choose~$k \in \setZ$ such that
  $\lambda^{k+1} < r < \lambda^k$. Let $A_1, \dots, A_{2^{mk}}$
  denote the connected components of~$\frC^m_{\lambda,k}$.
  
  We begin with~\ref{itm:cantor-estimates2}.  We estimate
  \begin{align*}
    \mathcal{L}^m\big( \set{
    \bar{x}: 
    d(\bar{x},\frC^m_\lambda) < r } \big)
    &\lesssim
      \mathcal{L}^m\big( \set{
      \bar{x}: 
      d(\bar{x},\frC^m_{\lambda,k}) < r } \big)
    \\
    &\lesssim
      \sum_{j=1}^{2^{mk}} \mathcal{L}^m\big( \set{
      \bar{x}: 
      d(\bar{x},A_j) < r } \big)
    \\
    &\leq 2^{mk}\, (\lambda^k + 2r)^m
    \\
    &= \lambda^{-\frD k}\, (\lambda^k + 2r)^m
    \\
    &\lesssim r^{m-\frD}
  \end{align*}
  using $2^m = \lambda^{-\frD}$. This proves~\ref{itm:cantor-estimates2}.

  Let us prove~\ref{itm:cantor-estimates1}. If
  $d(\bar{x}, \frC^m_\lambda)>r$, then
  $\mu^m_\lambda (B^m_r(\bar{x}))=0$. This explains the indicator
  function in~\ref{itm:cantor-estimates1}. Clearly,
  \begin{align*}
    \mu^m_\lambda (B^m_r(\bar{x}))=\sum_{l=1}^{2^{mk}}  \mu^m_\lambda (B^m_r(\bar{x}) \cap A_l).
  \end{align*}
  By the construction of $\frC^m_{\lambda,k}$ the sets~$A_j$ are pairwise disjoint translates
  of~$(0,\lambda^k)^m$. So the number of indices $l$ such that the intersection $B^m_r(\bar{x}) \cap A_l$ is nonempty does not exceed $2^m$. Thus
  \begin{align*}
  \mu^m_\lambda (B^m_r(\bar{x}))\leq 2^m\,\mu_\lambda^m(A_1)
  = 2^m\, 2^{-mk} = 2^m\, \lambda^{k \frD}   \lesssim r^{\frD}
  \end{align*}
  using again $2^m = \lambda^{-\frD}$ and $\lambda^k\leq r\lambda^{-1}$. This finishes the proof of~\ref{itm:cantor-estimates1}.

  Let us prove~\ref{itm:cantor-estimates3}. Note that
  \begin{align*}
    \big((\mu_\lambda^m \times \delta_0^{d-m}) *
    \indicator_{\set{\abs{\bar{x}}\leq \tau\abs{\hat{x}}}} \big)(x)
    &= \int_{\setR^m}
      \indicator_{\set{\abs{\bar{x}-\bar{y}}\leq
      \tau
      \abs{\hat{x}}}}(\bar{y})\,d\mu_\lambda^m(\bar{y}) =
      \mu^m_\lambda (B^m_{\tau \abs{\hat{x}}}(\bar{x})).
  \end{align*}
  Now, the claim follows by an application
  of~\ref{itm:cantor-estimates1} with $r=\tau \abs{\hat{x}}$.
\end{proof}

The following lemma will be useful to determine later the integrability of
our fractal examples.
\begin{lemma}
  \label{lem:cantor-weak} 
  Let~$\lambda \in (0,\frac 12)$, $1\leq m \leq d$ and
  \begin{align*}
  \frD := \dim(\frC_\lambda^m) = -m \log(2)/\log(\lambda).
  \end{align*}

We use
  the notation $x=(\bar{x},\hat{x}) \in \setR^m \times \setR^{d-m}$. Then
  \begin{enumerate}
    \setlength{\itemsep}{1ex}
  \item \label{itm:cantor-weak1}
    $\abs{\hat{x}}^{-\beta} \indicator_{\set{\abs{\bar{x}} \leq 4 \abs{\hat{x}}}} \in
    L^{\frac d\beta,\infty}(\Rd)$ for all~$0< \beta \leq d$.
  \item \label{itm:cantor-weak2}
    $\abs{\hat{x}}^{-\beta} \indicator_{\set{d(\bar{x},\frC^m_\lambda)
        \leq 4\, \abs{\hat{x}}}} \in L^{\frac{d-\frD}{\beta},\infty}(\Rd)$
    for all $0< \beta \leq d-\frD$.
  \end{enumerate}
\end{lemma}
\begin{proof}
  Let $\beta>0$. We begin with~\ref{itm:cantor-weak1}. Let
  $f := \abs{\hat{x}}^{-\beta} \indicator_{\set{\abs{\bar{x}} \leq 4
      \abs{\hat{x}}}}$ and~$\gamma>0$.  If $\abs{f}> \gamma$, then
  $\abs{\hat{x}} \leq \gamma^{-\frac 1 \beta}$. Thus,
  \begin{align*}
    \int_\Rd \gamma^{\frac d\beta} \indicator_{\set{\abs{f}> \gamma}}
    \,dx
    &\leq \int_\Rd \gamma^{\frac d\beta} \indicator_{\set{\abs{\hat{x}} \leq 
      \gamma^{-\frac{1}{\beta}}}}
      \indicator_{\set{\abs{\bar{x}} \leq 4 \abs{\hat{x}}}}
      \,dx
    \\
    &\lesssim \gamma^{\frac d\beta} \int_{\setR^m}  \indicator_{\set{\abs{\hat{x}} \leq 
      \gamma^{-\frac 1\beta}}} \abs{\hat{x}}^{m} \,d\hat{x}
    \\
    &\lesssim \gamma^{\frac{d}\beta-\frac{d}{\beta}} = 1.
  \end{align*}
  This proves~\ref{itm:cantor-weak1}.
  
  Now, let
  $g := \abs{\hat{x}}^{-\beta} \indicator_{\set{d(\bar{x},\frC^m_\lambda)
      \leq 4\, \abs{\hat{x}}}} \in L^{s,\infty}(\Rd)$ and~$\gamma>0$. If
  $\abs{g} > \gamma$, then $\abs{\hat{x}} \leq \gamma^{-\frac 1
    \beta}$. Thus, with
  Lemma~\ref{lem:cantor-estimates} we get
  \begin{align*}
    \int_\Rd \gamma^{\frac{d-\frD}\beta} \indicator_{\set{d(\bar{x},\frC^m_\lambda)
    \leq 4\, \abs{\hat{x}}}}
    \,dx
    &\leq \int_\Rd \gamma^{\frac{d-\frD}\beta} \indicator_{\set{\abs{\hat{x}} \leq 
      \gamma^{-\frac{1}{\beta}}}}
      \indicator_{\set{d(\bar{x},\frC^m_\lambda)
      \leq 4\, \abs{\hat{x}}}}
      \,dx
    \\
    &\lesssim \gamma^{\frac{d-\frD}\beta} \int_{\setR^m}  \indicator_{\set{\abs{\hat{x}} \leq 
      \gamma^{-\frac 1\beta}}} \abs{\hat{x}}^{m-\frD} \,d\hat{x}
    \\
    &\lesssim \gamma^{\frac{d-\frD}\beta-\frac{d-\frD}{\beta}} = 1.
  \end{align*}
  This proves~\ref{itm:cantor-weak2}.
\end{proof}

\begin{remark}
  \label{rem:cantor-weak}
  Note that integrability exponents $\frac{d}{\beta}$ and
  $\frac{d-\frD}{\beta}$ in Lemma~\ref{lem:cantor-weak} are sharp. In
  particular,
  $\abs{\hat{x}}^{-\beta} \indicator_{\set{\abs{\bar{x}} \leq 4
      \abs{\hat{x}}}} \notin L^{\frac d\beta}(\Omega)$ and
  $\abs{\hat{x}}^{-\beta} \indicator_{\set{d(\bar{x},\frC^m_\lambda)
      \leq 4\, \abs{\hat{x}}}} \notin
  L^{\frac{d-\frD}{\beta}}(\Omega)$ with $\Omega = (-1,1)^d$.
\end{remark}

\subsection{Construction of Fractal Examples}
\label{ssec:construction}

We can now construct our fractal examples, namely the functions~$u$
and~$b$. The contact set~$\frS$ will in our examples be a subset
of~$\setR^{d-1} \times \set{0}$ or $\set{0}^{d-1} \times \setR$. In
particular, we split~$\setR^d$ into~$\setR^{d-1} \times \setR$ and
write $x= (\bar{x},x_d) \in \setR^{d-1} \times \setR$.

We will provide some pictures
after the formal definition.
\begin{definition}[Fractal Examples]
  \label{def:fractal-examples}
  Let $\Omega := (-1,1)^d$ with~$d\geq 2$.  Let $u_d,A_d,b_d,$
  be as Definition~\ref{def:zhikov}. Let $1 < p_0 < \infty$. We
  define~$u,A,b$ on~$\overline{\Omega}$ distinguishing three cases:
  \begin{enumerate}
  \item (Matching the dimension; Zhikov) $p_0=d$:

    Let $u:= u_d$, $A := A_d$, $b := b_d$, $\frS := \set{0}$ and~$\frD
    := \dim \frS =0$.
  \item (Sub-dimensional) $1<p_0<d$:

    Let $\frS := \frC^{d-1}_\lambda \times \set{0}$ and
    $\frD= \dim(\frS)=\frac{(d-1)\log 2}{\log(1/\lambda)}$,
    where~$\lambda \in (0,\frac 12)$ is chosen such that
    $p_0 = d - \frD$.  Let $\rho \in C^\infty(\setR^d\setminus \frS)$
    be such that (using Lemma~\ref{lem:smooth-indicator})
    \begin{enumerate}[label={(\roman{*})}]
    \item
      $\indicator_{\set{d(\bar{x},\frC^{d-1}_\lambda) \leq 2 \abs{x_d}}}
      \leq \rho \leq \indicator_{\set{d(\bar{x},\frC^{d-1}_\lambda)
          \leq 4 \abs{x_d}}}$.
    \item
      $\abs{\nabla \rho} \lesssim \abs{x_d}^{-1}
      \indicator_{\set{2 \abs{x_d} \leq d(\bar{x},\frC^{d-1}_\lambda) \leq
          4 \abs{x_d}}}$.
    \end{enumerate}
    We define
    \begin{align*}
      u &:= \sgn(x_d)\, \rho(x),
      \\
      A &:= (\mu^{d-1}_\lambda \times \delta_0) * A_d,
      \\
      b &:= \divergence A.
    \end{align*}

  \item (Super-dimensional) $p_0 > d$:
    
    Let $\frS := \set{0}^{d-1} \times \frC_\lambda$ and
    $\frD= \dim(\frS)=\frac{\log 2}{\log(1/\lambda)}$,
    where~$\lambda \in (0,\frac 12)$ is chosen such that
    $p_0 = \frac{d-\frD}{1-\frD}$.  Let $\rho \in C^\infty(\setR^d \setminus \frS)$ be such that (using
    Lemma~\ref{lem:smooth-indicator})
    \begin{enumerate}[label={(\roman{*})}]
    \item
      $\indicator_{\set{d(x_d,\frC_\lambda) \leq 2
      \abs{\bar{x}}}}
      \leq \rho \leq \indicator_{\set{ d(x_d,\frC_\lambda) \leq
          4\abs{\bar{x}}}}$.
    \item
      $\abs{\nabla \rho} \lesssim \abs{\bar{x}}^{-1}
      \indicator_{\set{2\abs{\bar{x}} \leq d(x_d,\frC_\lambda) \leq
          4\abs{\bar{x}}}}$.
    \end{enumerate}
    We define
    \begin{align*}
      u &:= (\delta_0^{d-1} \times \mu_\lambda) * u_d
      \\
      A(x) &:=
             \frac{1}{\sigma_{d-1}} \abs{\bar{x}}^{1-d}
             \begin{pmatrix}
               0 & -\bar{x} \\
               \bar{x}^T & 0
             \end{pmatrix}
                            \rho(x)
      \\
      b &:= \divergence A.
    \end{align*}
  \end{enumerate}
\end{definition}
\begin{remark}
  We will use the functions~$u,b$ from the
  Definition~\ref{def:fractal-examples} with the exponent~$p_0$ later
  to construct a variable exponent $p\,:\, (-1,-1)^d \to (1,\infty)$
  with saddle point value~$p_0$ that provides a Lavrentiev gap,
  see~Subsection~\ref{ssec:variable-exponents}. This explains that we
  use~$p_0$ as a parameter to label our fractal examples. Another
  reason is that~$\nabla u$ is in the Marcinkiewicz (weak Lebesgue)
  space $L^{p_0,\infty}$ and $b \in L^{p_0',\infty}(\Omega)$, see
  Corollary~\ref{cor:ub-p0} and Remark~\ref{rem:ub-p0}.
\end{remark}

Let us provide a few pictures for the 2D case to illustrate our
fractal examples from Definition~\ref{def:fractal-examples}.  Using
the $\frac 13$-Cantor set~$\frC^1_{\frac 13}$ we provide in
Figure~\ref{fig:2Dsub}, resp. Figure~\ref{fig:2Dsup}, the
sub-dimensional, resp. super-dimensional case.  The case of matching the dimension
is already considered in Figure~\ref{fig:zhikov}. The left picture
shows the values of~$u$. The right picture shows the
component~$A_{1,2}$ of the skew-symmetric~$A$. The vertical and
horizontal patterns indicate that the function is almost linear along
those lines.
\begin{figure}[!ht]
  \centering
  
  \begin{tikzpicture}[scale=3]
    \draw[dashed] (-1,-1) -- (-1,+1) -- (+1,+1) -- (+1,-1) --cycle;

    \clip (-1,-1) rectangle (1,1);

    \fullpattern{%
      \fill[pattern=vertical lines] (-1/6,0) -- (0,1/12) -- (1/6,0) --
      (0,1/24) -- (-1/6,0) -- (0,-1/12) -- (1/6,0) -- (0,-1/24) --
      (-1/6,0);
      \draw (-1/6,0) -- (0,1/12) -- (1/6,0) --
      (0,1/24) -- (-1/6,0) -- (0,-1/12) -- (1/6,0) -- (0,-1/24) --
      (-1/6,0);
    }

    \node at (0,1/2) {$\frac 1 2$};
    \node at (0,-1/2) {$-\frac {1}{ 2}$};
    
    \node at (+5/6,0) {$0$};
    \node at (0,0) {$\scriptstyle 0$};
    \node at (-5/6,0) {$0$};
  \end{tikzpicture}
  \quad
  \begin{tikzpicture}[scale=3]
    \clip (-1,-1) rectangle (1,1);

    \filldraw[pattern=horizontal lines](-1,-1)--(-1,1)--(1,1)--(1,-1)--cycle;

    \fullpattern{%
      \filldraw[white] (-1/6,0) -- (0,1/3) -- (1/6,0) -- (0,-1/3) -- (-1/6,0); 
      \draw (-1/6,0) -- (0,1/3) -- (1/6,0) -- (0,-1/3) -- (-1/6,0); 

    }

    \node at (-3/4,0) {$-\frac 12$};
    \node at (3/4,0) {$\frac 12$};
    \node at (0,0) {$\scriptstyle 0$};
    \node at (-1/3,0) {\scalebox{0.6}{$\text{-}\frac 14$}};
    \node at (+1/3,0) {\scalebox{0.6}{$\frac 14$}};

    \draw[dashed] (-1,-1) -- (-1,+1) -- (+1,+1) -- (+1,-1) -- cycle;
  \end{tikzpicture}
  \caption{Sub-dimensional case: functions~$u$ and the
    component~$A_{1,2}$ for $d=2$, $p_0=2-\frD$ and
    $\frS = \frC^1_{\frac 13} \times \set{0}$. The lines indicate
    smooth almost linear extensions.}
  \label{fig:2Dsub}
\end{figure}
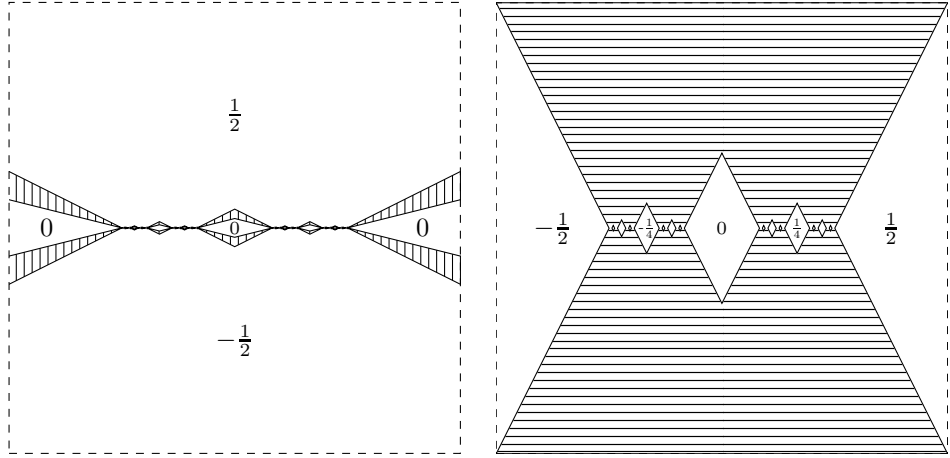
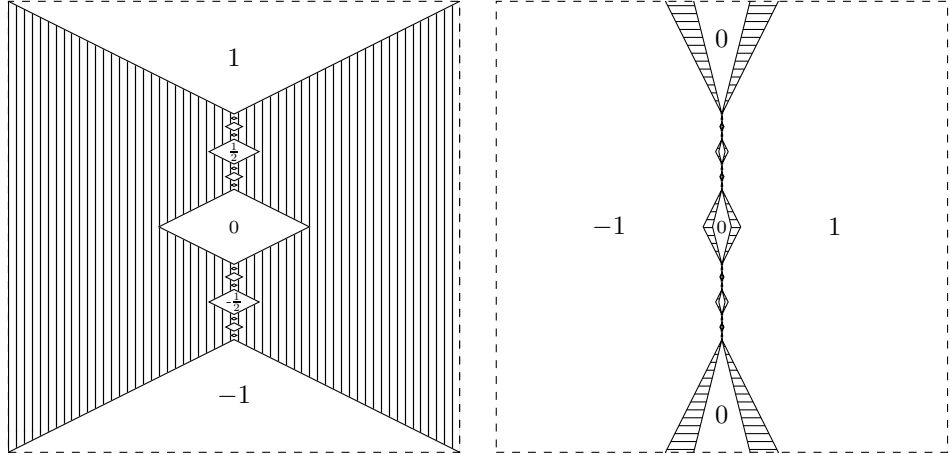
\begin{figure}[!ht]
  \centering
  
  \begin{tikzpicture}[scale=3]
    \draw[dashed] (-1,-1) -- (-1,+1) -- (+1,+1) -- (+1,-1) -- cycle;

    \clip (-1,-1) rectangle (1,1);
    
    \fill[pattern=vertical lines] (-1,-1) rectangle
    (1,1);
    
    \vfullpattern{%
      \filldraw[white] (0,-1/6) -- (1/3,0) -- (0,1/6) -- (-1/3,0) -- (0,-1/6); 
      \draw  (0,-1/6) -- (1/3,0) -- (0,1/6) -- (-1/3,0) -- (0,-1/6); 
    }

    \node at (0,-3/4) {$-1$};
    \node at (0,3/4) {$1$};
    
    \node at (0,0) {$\scriptstyle 0$};
    
    \node at (0,-1/3) {\scalebox{0.6}{$\text{-}\frac 12$}};
    \node at (0,+1/3) {\scalebox{0.6}{$\frac 12$}};

    \draw[dashed] (-1,-1) -- (-1,+1) -- (+1,+1) -- (+1,-1) -- cycle;
  \end{tikzpicture}
  \quad
  \begin{tikzpicture}[scale=3]
    \draw[dashed] (-1,-1) -- (-1,+1) -- (+1,+1) -- (+1,-1) --cycle;

    \clip (-1,-1) rectangle (1,1);
    
    \vfullpattern{%
      \fill[pattern=horizontal lines] (0,-1/6) -- (1/12,0) -- (0,1/6) --
      (1/24,0) -- (0,-1/6) -- (-1/12,0) -- (0,1/6) -- (-1/24,0) --
      (0,-1/6);
      \draw  (0,-1/6) -- (1/12,0) -- (0,1/6) --
      (1/24,0) -- (0,-1/6) -- (-1/12,0) -- (0,1/6) -- (-1/24,0) --
      (0,-1/6);
    }
    
    \node at (1/2,0) {$1$};
    \node at (-1/2,0) {$-1$};
    
    \node at (0,+5/6) {$0$};
    \node at (0,0) {$\scriptstyle 0$};
    \node at (0,-5/6) {$0$};
  \end{tikzpicture}

  \caption{Super-dimensional case: functions~$u$ and~$A_{1,2}$ for $d=2$,
    $p_0=\frac{d-\frD}{1-\frD}$ and $\frS = \set{0} \times \frC^1_{\frac 13} \times$. The
    lines indicate smooth almost linear extensions.}
  \label{fig:2Dsup}
\end{figure}

\begin{remark}
  \label{rem:avoid-diff-forms}
  The use of the skew-symmetric~$A$ allows us to avoid the language
  of differential forms:
  \begin{enumerate}
  \item  \label{itm:avoid-diff-forms1}
    For~$d=2$ we can rewrite~$A$ as
    \begin{align*}
      A &=
          \begin{pmatrix}
            0 & -v 
            \\
            v & 0
          \end{pmatrix}
    \end{align*}
    with $v\,:\, \setR^2 \to \setR$.  Then
    $\divergence A= (-\partial_2 v, \partial_1 v)^T = \nabla^\perp
    v$. Thus, $\divergence \divergence A = \divergence b =0$ becomes the well known
    $\divergence \nabla^\perp v = 0$, compare~\eqref{eq:divbd2} and
    Proposition~\ref{pro:b-divfree}.
  \item   \label{itm:avoid-diff-forms2}
    If~$d=3$ we can rewrite~$A$ as
    \begin{align*}
      A &:=
          \begin{pmatrix}
            0 & v_3 & -v_2
            \\
            -v_3 & 0 & v_1
            \\
            v_2 & -v_1 & 0
          \end{pmatrix}
    \end{align*}
    with $v\,:\, \setR^3 \to \setR^3$. Then $\divergence A= \curl
    v$. Thus, $\divergence \divergence A = \divergence b =0$ becomes
    the well known $\divergence \curl v = 0$. Hence, for $d=3$ we
    could also work with $v$ and $\curl v$ instead of $A$ and
    $\divergence A$. Compare also~\eqref{eq:divbd2} and
    Proposition~\ref{pro:b-divfree}.
  \end{enumerate}
\end{remark}
\begin{remark}
  \label{rem:rankine}
  If we use
  Remark~\ref{rem:avoid-diff-forms}~\ref{itm:avoid-diff-forms2} to
  find~$v$ with $A_3 = \curl v$. Then
  $(x_1,x_2) \mapsto v(x_1,x_2,x_3)$ is just a smooth version of the
  so called \emph{Rankine vortex}. It has a central core of
  radius~$\frac 12 \abs{x_3}$, where the velocity increases linearly,
  surrounded by a free vortex, where the velocity drops off from the
  center like~$\frac 1r$ with $r=\abs{(x_1,x_2)}$.


\end{remark}

\subsection{Properties of the Fractal Examples}
\label{ssec:prop-fract-exampl}

Let us derive a few useful properties of~$u$, $A$ and $b$.
\begin{proposition}
  \label{pro:prop-uAb}
  For~$1<p_0<\infty$ let $u,A,b$ be as is
  Definition~\ref{def:fractal-examples}. Then
  \begin{enumerate}
  \item \label{itm:prop-f1} $u \in L^\infty(\Omega) \cap W^{1,1}(\Omega) \cap C^\infty(\overline{\Omega}
  \setminus \frS)$,
  \item \label{itm:prop-f2} $A \in W^{1,1}(\Omega) \cap C^\infty(\overline{\Omega}
  \setminus \frS)$,
  \item \label{itm:prop-f3} $b \in L^1(\Omega) \cap
    C^\infty(\overline{\Omega} \setminus \frS)$. 
  \end{enumerate}
\end{proposition}
\begin{proof}
  The case~$p_0=d$ follows from Proposition~\ref{pro:prop-uAb-d}.  We
  continue with the sub-dimensional case $1<p_0 < d$. It is easy to
  see that
  $u \in C^\infty(\overline{\Omega} \setminus \frS) \cap
  L^\infty(\Omega)$.  Since $A_d \in C^\infty(\Rd \setminus \set{0})$
  and $\support (\mu_\lambda^{d-1} \times \delta_0) = \frS$, it
  follows from the definition by convolution that
  $A \in C^\infty(\overline{\Omega} \setminus \frS)$, so also
  $b \in C^\infty(\overline{\Omega} \setminus \frS)$. It also follows
  that $A \in W^{1,1}(\Omega)$ and $b \in L^1(\Omega)$. The
  case~$p_0>d$ is similar.
\end{proof}

\begin{proposition}
  \label{pro:est-uAb}
  For~$1<p_0<\infty$ let $u,A,b$ be as is
  Definition~\ref{def:fractal-examples}.
  \begin{enumerate}
  \item \label{itm:est-uAb1}
  \noindent%
  If $p_0=d$, then
  \begin{align*}
    \abs{\nabla u} 
    &\lesssim           \abs{\bar{x}}^{-1} \indicator_{ 
      \set{ 2 \abs{x_d} \leq \abs{\bar{x}} \leq 4
      \abs{x_d}}} \eqsim \abs{x_d}^{-1} \indicator_{ 
      \set{ 2 \abs{x_d} \leq \abs{\bar{x}} \leq 4
      \abs{x_d}}}
      ,
    \\
    \abs{b} 
    &\lesssim           \abs{x_d}^{1-d} \indicator_{ 
      \set{ 2 \abs{\bar{x}} \leq \abs{x_d} \leq 4
      \abs{\bar{x}}}} \eqsim \abs{\bar{x}}^{1-d} \indicator_{ 
      \set{ 2 \abs{\bar{x}} \leq \abs{x_d} \leq 4
      \abs{\bar{x}}}}
      .
    \\
    \intertext{If $1 < p_0 < d$, then}
    \abs{\nabla u} &\lesssim \abs{x_d}^{-1} \indicator_{
                     \set{
                     2 \abs{x_d} \leq 
                     d(\bar{x},\frC_\lambda^{d-1}) \leq 4 \abs{x_d}
                     }
                     }
                     ,
    \\
    \abs{b} 
                 &\lesssim \abs{x_d}^{\frD+1-d} \indicator_{ \set{ 
      d(\bar{x},\frC^{d-1}_\lambda) \leq \frac 12 \abs{x_d}}}
      .
    \\
    \intertext{If $p_0>d$, then}
    \abs{\nabla u} &\lesssim \abs{\bar{x}}^{\frD-1} \indicator_{ \set{ 
                     d(x_d,\frC_\lambda) \leq \frac 12 \abs{\bar{x}}}}
                     , 
    \\
    \abs{b}  &\lesssim \abs{ \bar{x}}^{1-d} \indicator_{
                             \set{ 2 \abs{\bar{x}} \leq d(x_d,\frC_\lambda) \leq 4
                             \abs{\bar{x}}}}. 
  \end{align*}
  \item \label{itm:est-uAb2} $\abs{\nabla u}\cdot \abs{b} = 0$ a.e. in~$\Omega$.
  \end{enumerate}
\end{proposition}
\begin{proof}
  We begin with~\ref{itm:est-uAb1}. The case~$p_0=d$ follows directly
  from Proposition~\ref{pro:prop-uAb-d}. We continue with the sub-dimensional
  case~$1<p_0 < d$.  It follows from the properties of~$\rho$ that
  \begin{align*}
    \abs{\nabla u} 
    &\lesssim \indicator_{ 
      \set{ 2 \abs{x_d} \leq
      d(\bar{x},\frC^{d-1}_\lambda) \leq 4
      \abs{x_d}}}
      \abs{x_d}^{-1}.
  \end{align*}
  If follows from
  \begin{align*}
    \abs{b}&= (\mu^{d-1}_\lambda \times \delta_0) *
                    b_d,
  \end{align*}
  Proposition~\ref{pro:prop-uAb-d} and Lemma~\ref{lem:cantor-estimates} that
  \begin{align*}
    \abs{b}&\leq
                    (\mu^{d-1}_\lambda \times \delta_0) *
                    \abs{b_d}
    \\
                  &\lesssim (\mu^{d-1}_\lambda \times \delta_0) * \big(
                    \indicator_{\set{\frac 14 \abs{x_d}
                    \leq \abs{\bar{x}} \leq \frac 12 \abs{x_d}}}\, \abs{x_d}^{1-d} \big)
    \\
                  &= \big((\mu^{d-1}_\lambda \times \delta_0) *
                    \indicator_{\set{\frac 14 \abs{x_d}
                    \leq \abs{\bar{x}} \leq \frac 12 \abs{x_d}}}\big)\,
                    \abs{x_d}^{1-d}
    \\
                  &\leq
                    \indicator_{ \set{ 
                    d(\bar{x},\frC^{d-1}_\lambda) \leq \frac 12 \abs{x_d}}}(x)
                    \,\abs{x_d}^{\frD+1-d}.
  \end{align*}
  This proves the sub-dimensional case.

  If remains to prove the super-dimensional case~$p_0 > d$.  It
  follows from the properties of~$\rho$ that
  \begin{align*}
     \abs{b} 
     &\lesssim \indicator_{ 
       \set{
       2 \abs{\bar{x}} \leq
       d(\bar{x_d},\frC_\lambda) \leq 4
       \abs{\bar{x}}
       }
       }
       \abs{\bar{x}}^{1-d}.
  \end{align*}
  If follows from
  \begin{align*}
    \nabla u&= (\delta_0^{d-1} \times \mu_\lambda)  *
                    \nabla u_d,
  \end{align*}
  Proposition~\ref{pro:prop-uAb-d} and Lemma~\ref{lem:cantor-estimates} that
  \begin{align*}
    \abs{\nabla u}&\leq
                    (\delta_0^{d-1} \times \mu_\lambda) *
                    \abs{\nabla u_d}
    \\
                  &\lesssim (\delta_0^{d-1} \times \mu_\lambda) * \big(
                    \indicator_{\set{\frac 14 \abs{\bar{x}}
                    \leq \abs{x_d} \leq \frac 12 \abs{\bar{x}}}}\, \abs{x_d}^{-1} \big)
    \\
                  &= \big((\delta_0^{d-1} \times \mu_\lambda) *
                    \indicator_{\set{\frac 14 \abs{\bar{x}}
                    \leq \abs{x_d} \leq \frac 12 \abs{\bar{x}}}}\big)\,
                    \abs{x_d}^{-1}
    \\
                  &\leq
                    \indicator_{ \set{ 
                    d(x_d,\frC_\lambda) \leq \frac 12 \abs{x_d}}}(x)
                    \,\abs{x_d}^{\frD-1}.
  \end{align*}
  This proves the super-dimensional case and concludes~\ref{itm:est-uAb1}. 

  The estimates in~\ref{itm:est-uAb1} immediately imply that the
  support of~$\nabla u$ and $b$ only overlaps at~$\frS$, which
  is a null set. This proves~\ref{itm:est-uAb2}.
\end{proof}
The following corollary clarifies the role of~$p_0$ in
Definition~\ref{def:fractal-examples}.
\begin{corollary}
  \label{cor:ub-p0}
  For~$1<p_0<\infty$ let $u,b$ be as in
  Definition~\ref{def:fractal-examples}. Then
  $\nabla u \in L^{p_0,\infty}(\Omega)$,
  $b \in L^{p_0',\infty}(\Omega)$.
\end{corollary}
\begin{proof}
  The proof is an immediate consequence of Lemma~\ref{lem:cantor-weak}
  in combination with
  Proposition~\ref{pro:est-uAb}~\ref{itm:est-uAb1}. For this recall
  that we have $p_0=d$ for matching the dimension, $p_0=d-\frD$ in the
  sub-dimensional case and $p_0 = \frac{d-\frD}{1-\frD}$ in the
  super-dimensional case. We apply Lemma~\ref{lem:cantor-weak} for
  $m=d-1$ and for~$m=1$ to cover all cases.
\end{proof}
\begin{remark}
  \label{rem:ub-p0}
  The integrability exponents of~$\nabla u$ and~$b$ are sharp. In
  particular, $\nabla u \notin L^{p_0}(\Omega)$ and
  $b \notin L^{p_0'}(\Omega)$. This can be shown with the help of
  Remark~\ref{rem:cantor-weak}.
\end{remark}

We also need localized version of~$u$, $A$ and $b$.
\begin{definition}
  \label{def:localuAb}
  For~$1<p_0<\infty$ let $u,A,b$ be as is
  Definition~\ref{def:fractal-examples}.  Let
  $\eta \in C^\infty_0(\Omega)$ with
  $\indicator_{(-\frac46,\frac46)^d} \leq \eta \leq
  \indicator_{(-\frac56,\frac56)^d}$ and $\norm{\nabla \eta}_\infty \leq
  c$. Then we define
  \begin{align*}
    u^\circ &= \eta u, &  u^\partial &= (1-\eta) u,
    \\
    A^\circ &= \eta A, &  A^\partial &= (1-\eta) A,
    \\
    b^\circ &= \divergence(A^\circ) = \divergence(\eta A),
                       &
                         b^\partial
                                     &= \divergence(A^\partial) = \divergence\big((1-\eta) A\big).
  \end{align*}
\end{definition}

The following proposition shows that $b$ and $b^\circ$ are divergence
free in the sense of distributions.

\begin{proposition}
  \label{pro:b-divfree}
  For all $w \in C^\infty(\Omega)$ we have
  \begin{align*}
    \int_\Omega  b^\circ \cdot \nabla w  \,dx = 0.
  \end{align*}
  For all $w \in C^\infty_0(\Omega)$ we have
  \begin{align*}
    \int_\Omega  b \cdot \nabla w  \,dx = 0.
  \end{align*}
  Moreover, $\divergence b = \divergence b^\circ = 0$ in the
  distributional sense and on~$\Omega \setminus \frS$ in the classical
  sense.
\end{proposition}
\begin{proof}
  For $w \in C^\infty(\overline{\Omega})$ we get by partial integration
  \begin{align*}
    \int_\Omega  b^\circ \cdot \nabla w  \,dx
    &=
      \int_\Omega   \divergence\big(\eta A\big) \cdot \nabla w\,dx =
      \int_\Omega  (\eta A) \cdot \nabla^2 w \,dx = 0,
  \end{align*}
  since~$A$ is anti-symmetric.  For
  $w \in C^\infty_0(\overline{\Omega})$ we also get by partial
  integration
  \begin{align*}
    \int_\Omega  b \cdot \nabla w  \,dx
    &=
      \int_\Omega   \divergence A \cdot \nabla w\,dx =
      \int_\Omega  A: \nabla^2 w \,dx = 0,
  \end{align*}
  since~$A$ is anti-symmetric. It follows that
  $\divergence b = \divergence b^\circ=0$ in the distributional
  sense. Since by Proposition~\ref{pro:prop-uAb} we have
  $b \in C^\infty(\overline{\Omega} \setminus \frS)$, it follows that
  $\divergence b= \divergence b^\circ = 0$ on $\Omega \setminus \frS$
  in the classical sense.
\end{proof}
Due to Proposition~\ref{pro:b-divfree} the functions~$b$ and $b^\circ$
are called~\emph{separating vector fields}.
\begin{proposition}
  \label{pro:int-bu}
  For~$1<p_0<\infty$ let $u,A,b$ be as is
  Definition~\ref{def:fractal-examples}. Then
  \begin{align*}
    \int_{\partial \Omega} (b \cdot \nu) u\,dS = 1.
  \end{align*}
\end{proposition}
\begin{proof}
  The case $p_0=d$ is already contained in
  Proposition~\ref{pro:match-int-bu}.

  Let us continue with the sub-dimensional case $1 < p_0 < d$. Note
  that $b=0$ on~$\partial \Omega$ except on the sets~$\set{x_d = \pm
    1} \cap \partial \Omega$. On these sets $u$ takes the
  values~$\pm \frac 12$ and $\nu = \pm e_d$. Moreover, $b_d$ is even
  with respect to~$x_d$. Thus, 
  \begin{align}
    \label{eq:int-bu1}
    \mathrm{I} := \int_{\partial \Omega} (b \cdot \nu) u\,dS 
    &=
      \int_{(-1,1)^{d-1}} b(\bar{x},1) \cdot e_d\,d\bar{x} =
      \int_{\setR^{d-1}} b(\bar{x},1) \cdot e_d\,d\bar{x}
  \end{align}
  using the first part of Lemma~\ref{lem:bd-int}. By definition of~$b$
  we have
  \begin{align*}
    b &= \divergence\big(\mu_\lambda^{d-1} \times \delta_0 * A_d\big)=
        (\mu_\lambda^{d-1} \times \delta_0) * \divergence(A_d) = 
        (\mu_\lambda^{d-1} \times \delta_0) * b_d.
  \end{align*}
  This,~\eqref{eq:int-bu1}, $\mu_\lambda^{d-1}(\setR^{d-1})=1$ and
  Lemma~\ref{lem:bd-int} imply
  \begin{align*}
    \mathrm{I} &=  \int_{\setR^{d-1}} b_d(\bar{x},1) \cdot e_d 
        \,d\bar{x} = 1.
  \end{align*}
  This proves the sub-dimensional case.

  Let us continue with the super-dimensional case $p_0 > d$. Note
  that $b=0$ on~$\partial \Omega$ except on the sets~$\set{x_d = \pm
    1} \cap \partial \Omega$. On these sets $u_d$ takes the
  values~$\pm \frac 12$ and $\nu = \pm e_d$. Moreover, $b_d\cdot e_d$ is even
  with respect to~$x_d$. Thus, 
  \begin{align}
    \label{eq:int-bu2}
    \mathrm{II} := \int_{\partial \Omega} (b \cdot \nu) u\,dS 
    &=
      \int_{(-1,1)^{d-1}} b(\bar{x},1) \cdot e_d\,d\bar{x}.
  \end{align}
  By definition
  of~$b$ we have
  \begin{align*}
    b &= \divergence(\rho A_d) = \divergence\big((\rho-1) A_d\big) +
        \divergence(A_d)
        = \divergence\big((\rho-1) A_d\big) + b_d.
  \end{align*}
  Let $g(\bar{x}) := (\rho(\bar{x})-1) e_d^T A_d(\bar{x},1)$. Then
  $g \in C^\infty_0((-1,1)^{d-1})$ and
  \begin{align*}
    b \cdot e_d = \divergence_{\bar{x}}
    g+ b_d \cdot e_d.
  \end{align*}
  Hence, by~\eqref{eq:int-bu2}, the theorem of Gau{\ss} and
  Lemma~\ref{lem:bd-int}
  \begin{align*}
    \mathrm{II}
    &= 
      \int_{(-1,1)^{d-1}} \divergence_{\bar{x}} g\,d\bar{x} +
      \int_{(-1,1)^{d-1}} b_d(\bar{x},1) \cdot e_d\,d\bar{x} = 0+1 =1.
  \end{align*}
  This proves the super-dimensional case.
\end{proof}
\begin{proposition}
  \label{pro:uAblocal}
  For~$1<p_0<\infty$ with the notation of
  Definition~\ref{def:localuAb} we have
  \begin{enumerate}
  \item $\support(u^\circ), \support(A^\circ), \support(b^\circ)
    \subset [-\frac56,\frac56]^d \compactsubset \Omega$.
  \item $u^\partial, A^\partial, b^\partial \in C^\infty_0(\overline{\Omega}
    \setminus \frS)$.
  \end{enumerate}
\end{proposition}
\begin{proof}
  The claim follows immediately from the definition,
  $\frS \compactsubset (-\frac46,\frac46)^d$ and
  Proposition~\ref{pro:prop-uAb}.
\end{proof}


\section{Important Consequences}
\label{sec:consequences}

Zhikov used the functions~$u_2$ and $b_2$ in order to derive the
Lavrentiev gap, $H\neq W$ and the different notions of $\px$-harmonic
functions. We show in this section that also our fractal examples
display these phenomena. We will do this in this section in
quite general form and apply it to specific examples in
Section~\ref{sec:applications}. 

\subsection{Energy and Generalized Orlicz Spaces}
\label{ssec:gener-orlicz-space}

In this section we introduce the necessary function spaces, the so
called generalized Orlicz and Orlicz-Sobolev spaces.

We assume that~$\Omega \subset \Rd$ is a domain of finite
measure\footnote{It is no problem to consider infinite domains, but it
  is not needed in this context of counter examples.}. Later in our
applications we will only use~$\Omega= (-1,1)^d$.

We say that~$\oldphi\,:\, [0,\infty) \to [0,\infty]$ is an Orlicz
function if $\oldphi$ is convex, left-continuous, $\oldphi(0)=0$,
$\lim_{t \to 0} \oldphi(t)=0$ and
$\lim_{t \to \infty} \oldphi(t)=\infty$.  The conjugate Orlicz
function~$\oldphi^*$ is defined by
\begin{align*}
  \oldphi^*(s) &:= \sup_{t \geq 0} \big( st - \oldphi(t)\big).
\end{align*}
In particular, $st \leq \oldphi(t) + \oldphi^*(s)$.

In the following we assume that
$\phi\,:\, \Omega \times [0,\infty) \to [0,\infty]$ is a generalized
Orlicz function, i.e. $\phi(x, \cdot)$ is an Orlicz function for
every~$x \in \Omega$ and $\phi(\cdot,t)$ is measurable for
every~$t\geq 0$. We define the conjugate function~$\phi^*$ pointwise,
i.e.  $\phi^*(x,\cdot) := (\phi(x,\cdot))^*$.

We further assume the following additional properties:
\begin{enumerate}
\item We assume that~$\phi$ satisfies the $\Delta_2$-condition,
  i.e. there exists~$c \geq 2$ such that for
  all~$x \in \Omega$ and all~$t\geq 0$
  \begin{align}
    \label{eq:phi-Delta2}
    \phi(x,2t) &\leq c\, \phi(x,t). 
  \end{align}
\item We assume that~$\phi$ satisfies the~$\nabla_2$-condition,
  i.e. $\phi^*$ satisfies the~$\Delta_2$-condition. As a consequence,
  there exist~$s>1$ and~$c> 0$ such that for all $x\in \Omega$,
  $t\geq 0$ and $\gamma \in [0,1]$ there holds
  \begin{align}
    \label{eq:phi-Nabla2}
    \phi(x,\gamma t) \leq c\,\gamma^s \,\phi(x,t).
  \end{align}
\item We assume that $\phi$ and~$\phi^*$ are proper, i.e. for
  every~$t\geq 0$ there holds $\int_\Omega \phi(x,t)\,dx< \infty$ and
  $\int_\Omega \phi^*(x,t)\,dx < \infty$.
\end{enumerate}
Let $L^0(\Omega)$ denote the set of measurable function on~$\Omega$
and $L^1_{\loc}(\Omega)$ denote the space of locally integrable 
functions. 
We define the generalized Orlicz norm by
\begin{align*}
  \norm{f}_{\phix} &:= \inf \biggset{\gamma > 0\,:\, \int_\Omega
                    \phi(x,\abs{f(x)/\gamma})\,dx \leq 1}.
\end{align*}

Then generalized Orlicz space~$L^{\phix}(\Omega)$ is defined as the
set of all measurable functions with finite generalized Orlicz norm
\begin{align*}
  L^{\phix}(\Omega) &:= \bigset{f \in L^0(\Omega)\,:\, \norm{f}_\phix<\infty
}.
\end{align*}

For example the generalized Orlicz function $\phi(x,t) = t^p$
generates the usual Lebesgue space~$L^p(\Omega)$.

The $\Delta_2$-condition of~$\phi$ and~$\phi^*$ ensures that our space
is uniformly convex. The condition that~$\phi$ and $\phi^*$ are proper
ensure that $L^{\phix}(\Omega) \embedding L^1(\Omega)$ and
$L^{\phidx}(\Omega) \embedding L^1(\Omega)$. Thus $L^{\phix}(\Omega)$
and $L^{\phidx}(\Omega)$ are Banach spaces.

We define the generalized Orlicz-Sobolev space~$W^{1,\phix}$ as
\begin{align*}
  W^{1,\phix}(\Omega) &:= \set{w \in W^{1,1}(\Omega)\,:\,
                        \nabla w \in L^{\phix}(\Omega)},
\end{align*}
with the norm
\begin{align*}
 \norm{w}_{1, \phi(\cdot)}:=\norm{w}_{1}+\norm{\nabla w}_{ \phi(\cdot)}.
\end{align*}

In general smooth functions are not dense
in~$W^{1,\phix}(\Omega)$. Therefore, we define~$H^{1,\phix}(\Omega)$
as
\begin{align*}
  H^{1,\phix}(\Omega) &:= \big(\text{closure of~$C^\infty(\Omega) \cap W^{1,\phix}(\Omega)$
                        in~$W^{1,\phix}(\Omega)$}\big).
\end{align*}
See \cite{DieHHR11} and~\cite{HarHas19} for further properties of
these spaces.

We also introduce the corresponding spaces with zero boundary values
as
\begin{align*}
  W^{1,\phix}_0(\Omega) &:= \set{w \in W^{1,1}_0(\Omega)\,:\,
                        \nabla w \in L^{\phix}(\Omega)}
\end{align*}
with same norm as in~$W^{1,\phi(\cdot)}(\Omega)$. And the corresponding space of smooth functions is defined as  
\begin{align*}
  H_0^{1,\phix}(\Omega) &:= \big(\text{closure of~$C^\infty_0(\Omega) \cap W^{1,\phix}(\Omega)$
                        in~$W^{1,\phix}(\Omega)$}\big). 
\end{align*}
The space~$W_0^{1,\phix}(\Omega)$ are exactly those function, which can
be extended by zero to~$W^{1,\phix}(\Rd)$ functions.

Let us define our
energy~$\mathcal{F}\,:\, W^{1,\phix}(\Omega) \to \setR$ by
\begin{align*}
  \mathcal{F}(w) &:= \int_\Omega \phi(x,\abs{\nabla w(x)})\,dx.
\end{align*}
In the language of function spaces~$\mathcal{F}$ is a semi-modular
on~$W^{1,\phix}(\Omega)$ and a modular on~$W^{1,\phix}_0(\Omega)$.

\subsection{\texorpdfstring{H $\neq$ W and H$_0 \neq$ W$_0$}{H != W
    and H0 != W0}}
\label{ssec:HneqW}

In this section we show how to use the function~$u$ and the
vector~$b$ from Definition~\ref{def:fractal-examples} to give examples
for~$W^{1,\phix}(\Omega) \neq H^{1,\phix}(\Omega)$ and
for~$W_0^{1,\phix}(\Omega) \neq H^{1,\phix}_0(\Omega)$. In this
section, we need the following assumption:
\begin{assumption}
  \label{ass:HneqW}
  Let $u,u^\circ, u^\partial,b, b^\circ,b^\partial$ be  as in
  Section~\ref{sec:constr-fract}, i.e.
  Proposition~\ref{pro:prop-uAb}, \ref{pro:est-uAb}, \ref{pro:int-bu}
  and~\ref{pro:uAblocal} hold. Let $\phi$ be such that
  $u \in W^{1,\phix}(\Omega)$ and $b \in L^{\phidx}(\Omega)$.
\end{assumption}
For all $w \in W^{1,\phix}(\Omega)$ we define the continuous functionals
\begin{align}
  \label{eq:defS}
  \begin{aligned}
    \mathcal{S}(w) &:= \int_\Omega b \cdot \nabla w \,dx,
    \\
    \mathcal{S}^\circ(w) &:= \int_\Omega b^\circ \cdot \nabla w \,dx,
    \\
    \mathcal{S}^\partial(w) &:= \int_\Omega b ^\partial \cdot \nabla w
   \,dx.
  \end{aligned}
\end{align}
This is well defined, since
$b, b^\circ, b^\partial \in L^{\phidx}(\Omega)$.
\begin{proposition}
  \label{pro:separating}
  For all $w \in H^{1,\phix}(\Omega)$  we have
  $\mathcal{S}^\circ(w)=0$. Moreover, for all $w \in
  H^{1,\phix}_0(\Omega)$ we have~$\mathcal{S}(w)=0$. 
\end{proposition}
\begin{proof}
  Due to Proposition~\ref{pro:b-divfree} we have
  $\mathcal{S}^\circ(w)=0$ for~$w \in C^\infty(\Omega)$ and
  $\mathcal{S}(w)=0$ for~$w \in C^\infty_0(\Omega)$. Now,
  the claim follows by density.
\end{proof}
Due to Proposition~\ref{pro:separating} the functionals~$\mathcal{S}$
and~$\mathcal{S}^\circ$ are called \emph{separating functionals}.
\begin{proposition}
  \label{pro:Su}
  There holds
  \begin{enumerate}
  \item  $\mathcal{S}(u) = 0$, $\mathcal{S}(u^\partial) = 1$ and $\mathcal{S}(u^\circ) =
    -1$.
  \item $\mathcal{S^\partial}(u) = 1$,
    $\mathcal{S^\partial}(u^\partial) = 1$ and
    $\mathcal{S^\partial}(u^\circ) = 0$.
  \item  $\mathcal{S}^\circ(u) = -1$, $\mathcal{S}^\circ(u^\partial) =
    0$ and $\mathcal{S}^\circ(u^\circ) = 
    -1$.  
  \end{enumerate}
\end{proposition}
\begin{proof}
  Since $\nabla u \cdot b=0$ everywhere, we have $\mathcal{S}(u)=0$.
  Since
  $u^\partial, v^\partial \in C^\infty_0(\overline{\Omega} \setminus
  \frS)$, we can use partial integration to get
  \begin{align*}
    \mathcal{S}^\partial(u^\partial) &= \int_\Omega b^\partial\cdot  \nabla u^\partial  \,dx
    = \int_{\partial \Omega} (b^\partial \cdot \nu) u^\partial\,ds
      = \int_{\partial \Omega} (b \cdot \nu) u\,ds = 1
  \end{align*}
  using also Proposition~\ref{pro:int-bu}.  Since
  $u^\partial \in C^\infty_0(\overline{\Omega} \setminus \frS)$,
  $b^\circ \in C^\infty(\overline{\Omega} \setminus \frS)$,
  $\support b^\circ \compactsubset \Omega$ and $\divergence b^\circ =0$ on~$\Omega \setminus \frS$, we can use
  partial integration together to get
  \begin{align*}
    \mathcal{S}^\circ(u^\partial) = \int_\Omega b^\circ \cdot \nabla u^\partial
     \,dx = -
    \int_{(-1,1)^d \setminus \frS} \divergence b^\circ \cdot  u^\partial  \,dx = 0.
  \end{align*}
  Analogously, we obtain $\mathcal{S}^\partial(u^\circ)=0$. Now,
  \begin{align*}
    \mathcal{S}^\circ(u^\circ) &= 
    \mathcal{S}(u) -
    \mathcal{S}^\partial(u^\partial) -
    \mathcal{S}^\circ(u^\partial) - 
    \mathcal{S}^\partial(u^\circ) = 0 - 1 - 0 - 0 = -1.
  \end{align*}
  This proves the claim.
\end{proof}
We come to the main result of this subsection.
\begin{theorem}[H $\neq$ W]
  \label{thm:HneqW}
  Under the assumption~\ref{ass:HneqW} there holds
  \begin{enumerate}
  \item $u^\circ \in W^{1,\phix}(\Omega) \setminus
    H^{1,\phix}(\Omega)$.
  \item $u^\circ \in W_0^{1,\phix}(\Omega) \setminus
    H_0^{1,\phix}(\Omega)$.
  \end{enumerate}
\end{theorem}
\begin{proof}
  Recall that
  $u^\circ \in W^{1,\phix}_0(\Omega) \subset W^{1,\phix}(\Omega)$. Due
  to Proposition~\ref{pro:separating} we
  know that~$\mathcal{S}^\circ=0$ on $H^{1,\phix}(\Omega)$ and therefore
  also on
  $H_0^{1,\phix}(\Omega)$. However, 
  $\mathcal{S}^\circ(u^\circ) = -1$ by proposition~\ref{pro:Su}. This proves
  $u^\circ \notin H^{1,\phix}(\Omega)$ and
  $u^\circ \notin H^{1,\phix}_0(\Omega)$. This proves the claim.
\end{proof}

\subsection{Lavrentiev Gap}
\label{ssec:lavrentiev-gap}

In this section we show how to use the function~$u$ and the vector
field~$b$ from
Definition~\ref{def:fractal-examples} for the Lavrentiev gap. In this
section, we need the following assumption:
\begin{assumption}
  \label{ass:gap}
  Let $u,u^\circ, u^\partial,b, b^\circ,b^\partial$ as in
  Section~\ref{sec:constr-fract}, i.e.
  Proposition~\ref{pro:prop-uAb}, \ref{pro:est-uAb}, \ref{pro:int-bu}
  and~\ref{pro:uAblocal} hold. Let $\phi$ be such that
  $u \in W^{1,\phix}(\Omega)$ and $b \in L^{\phidx}(\Omega)$. Also
  recall, that $\phi^*$ satisfies the~$\Delta_2$-condition.
\end{assumption}
From the~$\Delta_2$-condition of~$\phi^*$, see~\eqref{eq:phi-Nabla2}, it follows that
\begin{align}
  \label{eq:Fdelta2}
  \lim_{t \searrow 0} \frac{\mathcal{F}(t w)}{t}
  &\leq
    \lim_{t \searrow 0} \bigg(   c\, t^{s-1}\int_\Omega \phi(x, \abs{\nabla
    w(x)})\,dx \bigg) =0.
\end{align}
We come to the main result of this subsection.
\begin{theorem}[Lavrentiev gap]
  \label{thm:gap}
  Under the assumptions~\ref{ass:gap} the functional
  \begin{align*}
    \mathcal{G} := \mathcal{F} + \mathcal{S}^\circ
  \end{align*}
  with~$\mathcal{S}^\circ$ defined in~\eqref{eq:defS} has a Lavrentiev gap, i.e.
  \begin{align*}
    \inf \mathcal{G}(W^{1,\phix}_0(\Omega)) < \inf
    \mathcal{G}(H^{1,\phix}_0(\Omega))=0.
  \end{align*}
\end{theorem}
\begin{proof}
  Due to Proposition~\ref{pro:separating} we have
  $\mathcal{G} = \mathcal{F}$ on~$H^{1,\phix}_0(\Omega)$, which
  implies that $\inf \mathcal{G}(H^{1,\phix}_0(\Omega))= 0$. However,
  for~$t >0$ we have
  \begin{align*}
    \mathcal{G}(t u^\circ) &= \mathcal{F}(t u^\circ) + t\,
                             \mathcal{S}^\circ(u^\circ) = t\, \bigg(
                             \frac{\mathcal{F}(tu^\circ)}{t} - 1 \bigg)
  \end{align*}
  using~$\mathcal{S}^\circ(u^\circ)=-1$ by Proposition~\ref{pro:Su}.
  Since $\lim_{t\to 0} \frac{\mathcal{F}(t u^\circ)}{t} = 0$
  by~\eqref{eq:Fdelta2}, the right-hand side becomes negative for
  small~$t>0$. Thus $\inf \mathcal{G}(W^{1,\phix}_0(\Omega)) < 0$.
\end{proof}

\subsection{\texorpdfstring{H-harmonic $\neq$ W-harmonic}{H-harmonic != W-harmonic}}
\label{ssec:h-harmonic-neq}

In this section we show that the spaces~$W^{1,\phix}(\Omega)$ and
$H^{1,\phix}(\Omega)$ lead to different concepts of $\phix$-harmonic functions.

Let us start by introducing spaces with boundary values: for
$g \in H^{1,\phix}(\Omega)$ we define
\begin{align*}
  H_g^{1,\phix}(\Omega) &:= g + H_0^{1,\phix}(\Omega).
\end{align*}
For $g \in W^{1,\phix}(\Omega)$ we define
\begin{align*}
  W_g^{1,\phix}(\Omega) &:= g + W_0^{1,\phix}(\Omega).
\end{align*}
Since $u^\partial \in W^{1,\phix}(\Omega)$, we can define
\begin{align*}
  h_W &= \argmin \mathcal{F}\big(W^{1,\phix}_{u^\partial}(\Omega)\big).
\end{align*}
Formally, it satisfies the Euler-Lagrange equation (in the weak sense)
\begin{align*}
  -\Delta_\phix h_W := -\divergence \bigg( \frac{\phi'(x,\abs{\nabla h_W})}{\abs{\nabla
  h_W}} \nabla h_W \bigg) &= 0 \qquad \text{ in $
                            (W^{1,\phix}_0(\Omega))^*$},
\end{align*}
where $\phi'(x,t)$ is the derivative with respect to~$t$.
However, since also $u^\partial \in  H^{1,\phix}(\Omega)$, we can define
\begin{align*}
  h_H &= \argmin \mathcal{F}\big(H^{1,\phix}_g(\Omega)\big).
\end{align*}
Then
\begin{align*}
  -\Delta_\phix h_W :=-\divergence \bigg( \frac{\phi'(x,\abs{\nabla h_H})}{\abs{\nabla
  h_H}} \nabla h_H \bigg) &= 0 \qquad \text{ in $
                            (H^{1,\phix}_0(\Omega))^*$}.
\end{align*}
If $\phi(x,t)=\frac 12 t^2$, then $\Delta_\phix$ is just the standard Laplacian.
If $\phi(x,t)=\frac 1p t^p$, then $\Delta_\phix$ is the $p$-Laplacian.

Thus $h_W$ and $h_H$ are both~$\phix$-harmonic but $h_W$ is
$\phix$-harmonic in the sense of~$W^{1,\phix}$ and $h_H$ is
$\phix$-harmonic with respect to~$H^{1,\phix}$.
Our goal is to provide an example, where these concepts differ. For
this we assume the following:
\begin{assumption}
  \label{ass:harmonic}
  Let $u,u^\circ, u^\partial,b, b^\circ,b^\partial$ as in
  Section~\ref{sec:constr-fract}, i.e.
  Proposition~\ref{pro:prop-uAb}, , \ref{pro:est-uAb}, \ref{pro:int-bu}
  and~\ref{pro:uAblocal} hold. Let $\phi$ be such that
  $u \in W^{1,\phix}(\Omega)$ and $b \in L^{\phidx}(\Omega)$.

  Moreover, assume that there exists $s,t > 0$ such that
  \begin{align}
    \label{eq:ass_FF*}
    \mathcal{F}(tu) + \mathcal{F}^*(s b) &< t s,
  \end{align}
  where
  \begin{align*}
    \mathcal{F}^*(g) &:= \int_\Omega \phi^*(x\, \abs{g(x)})\,dx.  
  \end{align*}
\end{assumption}
We come to the main result of this subsection.
\begin{theorem}[H-harmonic $\neq$ W-harmonic]
  \label{thm:harmonic}
  Under the Assumption~\ref{ass:harmonic} there
  exists~$g \in H^{1,\phix}(\Omega)$ such that the $\phix$-harmonic
  functions~$h_W$ in the sense of~$W^{1,\phix}$ and the
  $\phix$-harmonic function~$h_H$ in the sense of~$H^{1,\phix}$ with
  the same boundary values~$g$ differ. In particular, for
  \begin{align*}
    h_W &= \argmin
          \mathcal{F}\big(W^{1,\phix}_{u^\partial}(\Omega)\big) \qquad
    \\
    h_H &= \argmin \mathcal{F}\big(H^{1,\phix}_{u^\partial}(\Omega)\big)
  \end{align*}
  we have $h_W \neq h_H$ and $\mathcal{F}(h_W) < \mathcal{F}(h_H)$.
\end{theorem}
\begin{proof}
  We define $g := t u^\partial \in H^{1,\phix}(\Omega)$, with~$t>0$ to
  be chosen later. Now, let
  \begin{align*}
    w_t &:= \argmin \mathcal{F}(W^{1,\phix}_{tu^\partial}(\Omega)),
    \\
    h_t &:= \argmin \mathcal{F}(H^{1,\phix}_{tu^\partial}(\Omega)).
  \end{align*}
  We have  $t u = tu^\partial + tu^\circ \in W_{t
    u^\partial}^{1,\phix}(\Omega)$. Thus,
  \begin{align}
    \label{eq:Fwt}
    \mathcal{F}(w_t) &\leq \mathcal{F}(t u).
  \end{align}
  From the other hand, using Young's inequality,  we get for all~$s>0$ that 
  \begin{align*}
    \mathcal{F}(h_t)
    &= \int_\Omega \phi(x,\abs{\nabla h_t})\,dx
    \\
    &\geq \int_\Omega  \nabla h_t \cdot (s b)\,dx - \phi^*(x,
      s \abs{b}) \big)\,dx
    \\
    &= s\,\mathcal{S}(h_t) - \mathcal{F}^*(s b).
  \end{align*}
  Since $h_t - tu^\partial \in H_0^{1,\phix}(\Omega)$, we have
  $\mathcal{S}(h_t - tu^\partial)=0$ by
  Theorem~\ref{pro:separating}. This and
  $\mathcal{S}(u^\partial)=1$ by Proposition~\ref{pro:Su} imply
  \begin{align}
    \label{eq:Fht}
    \mathcal{F}(h_t)
    &= s\,\mathcal{S}(t u^\partial) - \mathcal{F}^*(s b)
    = t s - \mathcal{F}^*(s b).
  \end{align}
  Combining~\eqref{eq:Fwt} and~\eqref{eq:Fht} we get
  \begin{align*}
    \mathcal{F}(h_t) - \mathcal{F}(w_t) \geq t s -
    \mathcal{F}(tu) - \mathcal{F}^*(s b)
  \end{align*}
  for all $t,s > 0$. By Assumption~\ref{eq:ass_FF*} we can
  find~$t,s>0$ such that the right hand-side of last inequality is positive. For
  these~$t,s$ we have $\mathcal{F}(h_t) > \mathcal{F}(w_t)$. This
  proves the claim for $h_H := h_t$ and $h_W := w_t$.
\end{proof}


\section{Applications}
\label{sec:applications}

We will now apply our results to the following three models:
\begin{tabbing}
  Variable exponent space: \qquad\= $\phi(x,t) = \tfrac{1}{p(x)} t^{p(x)}$.
  \\[2mm]
  Double phase potential: \> $\phi(x,t) =  \tfrac 1p t^p + a(x) \tfrac
  1q t^q = \tfrac 1p t^p + \tfrac 1q
  (\omega(x)t)^q$.
  \\[2mm]
  Weighted $p$-energy: \> $\phi(x,t) = \tfrac 1p a(x) t^p = \tfrac 1p (\omega(x) t)^p$.
\end{tabbing}

\subsection{Variable Exponents}
\label{ssec:variable-exponents}

In this section we study the variable exponent model. In particular,
we assume that
\begin{align*}
  \phi(x,t) = \tfrac{1}{p(x)} t^{p(x)},
\end{align*}
where $p\,:\, \Omega \to (1,\infty)$ is a variable exponent. The
corresponding energy is
\begin{align*}
  \mathcal{F}(w) &= \int_\Omega \tfrac{1}{p(x)} \abs{\nabla w}^{p(x)}\,dx.
\end{align*}
We abbreviate $W^{1,\px}(\Omega) := W^{1,\phix}(\Omega)$ and
similarly~$W^{1,\px}_0$, $H^{1,\px}$ and $H^{1,\px}_0$. 

Our main result of the variable exponent model is the following:
\begin{theorem}
  \label{thm:main-px}
  Let~$\Omega=(-1,1)^d$. Let $1< p^- < p^+ < \infty$. Then there
  exists a variable exponent~$p\,:\, \Omega \to [p^-,p^+]$ such that
  \begin{enumerate}
  \item $H^{1,\px}(\Omega) \neq W^{1,\px}(\Omega)$ and
    $H_0^{1,\px}(\Omega) \neq W_0^{1,\px}(\Omega)$.
  \item There exists a linear, continuous
    functional~$\mathcal{S}^\circ\,:\, W^{1,\px}(\Omega) \to \setR$ such
    that functional~$\mathcal{G} := \mathcal{F} + \mathcal{S}^\circ$
    has a Lavrentiev gap, i.e.
    \begin{align*}
      \inf \mathcal{G}\big(W^{1,\px}_0(\Omega) \big)
      &<
        \inf \mathcal{G}\big(H^{1,\px}_0(\Omega)\big) =0.
    \end{align*}
  \item The notions of $\px$-harmonic functions with respect to
    $W^{1,\px}$ and $H^{1,\px}$ differ.
  \end{enumerate}
\end{theorem}
\begin{proof}
  Choose $p_0$ with $p^- < p_0 < p^+$. Now, let~$u,b$ be as in
  Definition~\ref{def:fractal-examples} and $\mathcal{S}^\circ$ as
  in~\eqref{eq:defS}.
  We begin with the definition of the variable exponent~$p$.
  \begin{enumerate}
  \item (Matching the dimension; Zhikov) $p_0=d$: Define
    \begin{align*}
      p(x) &:=
             \begin{cases}
               p^- &\quad \text{for $\abs{x_d} \leq \abs{\bar{x}}$},
               \\
               p^+ &\quad \text{for $\abs{x_d} > \abs{\bar{x}}$}.
             \end{cases}
    \end{align*}
  \item (Sub-dimensional) $1<p_0<d$: Define
    \begin{align*}
      p(x) &:=
             \begin{cases}
               p^- &\quad \text{for $\abs{x_d} \leq d(\bar{x},\frC^{d-1}_\lambda)$},
               \\
               p^+ &\quad \text{for $\abs{x_d} > d(\bar{x},\frC^{d-1}_\lambda)$}.
             \end{cases}
    \end{align*}
  \item (Super-dimensional) $p_0 > d$:
    Define
    \begin{align*}
      p(x) &:=
             \begin{cases}
               p^- &\quad \text{for
                 $d(x_d,\frC_\lambda) \leq \abs{\bar{x}}$},
               \\
               p^+ &\quad \text{for
                 $d(x_d,\frC_\lambda) > \abs{\bar{x}}$}.
             \end{cases}
    \end{align*}
  \end{enumerate}
  In particular, it follows from Proposition~\ref{pro:est-uAb} that
  \begin{align}
    \label{eq:supp-p+-}
    \set{\nabla u \neq 0} \subset \set{p=p^-} \quad \text{and} \quad
    \set{b \neq 0} \subset \set{p=p^+}.
  \end{align}
  Thus, with Corollary~\ref{cor:ub-p0} and $p^- < p_0 < p^+$ we obtain
  \begin{align*}
    \norm{\nabla u}_{L^{\px}(\Omega)}
    &\lesssim
      \norm{\nabla u}_{L^{p^-}(\Omega)} \lesssim
      \norm{\nabla u}_{L^{p_0,\infty}(\Omega)} < \infty
    \\
    \norm{b}_{L^{\pdx}(\Omega)}
    &\lesssim
      \norm{b}_{L^{(p^+)'}(\Omega)} \lesssim
      \norm{b}_{L^{p_0,\infty}(\Omega)} < \infty.
  \end{align*}
  This proves $u \in W^{1,\px}(\Omega)$ and $b \in
  L^\pdx(\Omega)$. This proves the validity of Assumption~\ref{ass:HneqW}.

  Since $1 < p^- \leq p \leq p^+ < \infty$, it follows that $\phi$
  and~$\phi^*$ satisfy the~$\Delta_2$ condition. Thus
  Assumption~\ref{ass:gap} also holds.

  Using~\eqref{eq:supp-p+-} we obtain for all~$s,t>0$ 
  \begin{align*}
    \mathcal{F}(tu) +
    \mathcal{F}^*(sb)
    &= t^{p^-}
      \mathcal{F}(u) + s^{(p^+)'}
      \mathcal{F}^*(b)
    \\
    &= \tfrac{1}{p_0} t^{p_0} \big( t^{p^--p_0} p_0
      \mathcal{F}(u) \big) + \tfrac{1}{p_0'} s^{p_0'} \big(
      s^{(p^+)'-p_0'} p_0'
      \mathcal{F}^*(b)  \big).
  \end{align*}
  Now, fix $s:= t^{p_0}-1$. Then for suitable large~$t$ (and therefore
  large~$s$) we obtain
  \begin{align*}
    \mathcal{F}(tu) +
    \mathcal{F}^*(sb)
    &\leq \tfrac{1}{p_0} t^{p_0} \cdot \tfrac 12 +
                \tfrac{1}{p_0'} s^{p_0'} \tfrac 12 = \tfrac 12 ts < ts.
  \end{align*}
  This proves Assumption~\ref{ass:harmonic}.

  Overall, we have constructed $u$, $b$, and $p$ such that the
  Assumptions~\ref{ass:HneqW}, \ref{ass:gap} and~\ref{ass:harmonic}
  holds.  Now, the claim follows from the results of
  Theorem~\ref{thm:HneqW}, \ref{thm:gap} and~\ref{thm:harmonic} of
  Section~\ref{sec:consequences}.
\end{proof}

The exponent in Theorem~\ref{thm:main-px} was discontinuous at the
singular set~$\frS$. The following result shows that it is also
possible to construct a uniformly continuous exponent with the same
phenomena.  However, this exponent is not $\log$-H\"older continuous,
since this would imply the $H^{1,\px}=W^{1,\px}$ by means of
convolution, see \cite{Zhi95} and~\cite[Section~4.6]{DieHHR11}.
\begin{theorem}
  \label{thm:main-px-cont}
  Let~$\Omega=(-1,1)^d$ with~$d \geq 2$. Let $1< p_0 < \infty$. Let
  $\frS$ (the fractal contact set) be as in
  Definition~\ref{def:fractal-examples}.  Then there exists a
  uniformly continuous variable exponent~$p$ with saddle points
  on~$\frS$ and $p=p_0$ on~$\frS$ and
  \begin{enumerate}
  \item $H^{1,\px}(\Omega) \neq W^{1,\px}(\Omega)$ and
    $H_0^{1,\px}(\Omega) \neq W_0^{1,\px}(\Omega)$.
  \item There exists a linear, continuous
    functional~$\mathcal{S}^\circ\,:\, W^{1,\px}(\Omega) \to \setR$ such
    that functional~$\mathcal{G} := \mathcal{F} + \mathcal{S}^\circ$
    has a Lavrentiev gap, i.e.
    \begin{align*}
      \inf \mathcal{G}\big(W^{1,\px}_0(\Omega) \big)
      &<
        \inf \mathcal{G}\big(H^{1,\px}_0(\Omega)\big) =0.
    \end{align*}
  \item The notions of $\px$-harmonic functions with respect to
    $W^{1,\px}$ and $H^{1,\px}$ differ.
  \end{enumerate}
\end{theorem}
\begin{proof}
  The proof is similar to the one of Theorem~\ref{thm:main-px}. So we
  only point out the difference. Let~$u,b$ be as in
  Definition~\ref{def:fractal-examples} and $\mathcal{S}^\circ$ as
  in~\eqref{eq:defS}.  We have to show that
  $\nabla u \in L^\px(\Omega)$ and $b \in L^\pdx(\Omega)$ and to
  verify Assumption~\ref{ass:harmonic}.

  Let $\sigma(t) := \frac{1}{(\log(e+1/t))^\kappa}$ with
  $0<\kappa< 1$. Then for any~$c_1 > 0$ there holds
  \begin{align}
    \label{eq:p-cont-aux}
    \begin{aligned}
      \int_0^1 t^{-1+c_1 \sigma(t)}\,dt &= \int_0^1 \exp \bigg(-\log t +
      \frac{c_1}{(\log(e+1/t))^{\kappa-1}} \bigg)\,dt
      \\
      &\leq c + \int_0^{1/e} \exp \bigg(-\log t +
      \frac{c_1\,c}{\abs{\log(t)}^{\kappa-1}} \bigg)\,dt
      \\
      &= c + \int_1^\infty \exp\big( - c_1\,c\, r^{1-\kappa}\big)\,dr
      < \infty.
    \end{aligned}
  \end{align}
  \begin{enumerate}
  \item 
    We begin with the case of matching the dimension~$p_0=d$.
    Let~$\theta_p \in C^\infty_0((0,\infty))$ be such that
    $\indicator_{(2,\infty)} \leq \theta_p \leq \indicator_{(\frac 12,
      \infty)}$ and $\norm{\theta_p'}_\infty \le 6$ and define
    \begin{alignat*}{2}
      &p^-(x)& &:= p_0 - \sigma(\abs{x_d}) = d - \sigma(\abs{x_d}),
      \\
      &p^+(x) &&:= p_0 + \sigma(\abs{x_d}) = d + \sigma(\abs{x_d}),
      \\
      &p(x) &&:= p^-(x) \bigg( 1 - \theta_p\bigg(\frac{\abs{x_d}}{\abs{\bar{x}}}\bigg)
      \bigg) + p^+(x)\theta_p\bigg(\frac{\abs{x_d}}{\abs{\bar{x}}}\bigg).
    \end{alignat*}
    Then~$p$ is uniformly continuous with modulus of
    continuity~$\sigma$. It follows by Proposition~\ref{pro:est-uAb},
    Lemma~\ref{lem:cantor-estimates},~\eqref{eq:p-cont-aux} and $p_0=d$
    that
    \begin{align}
      \label{eq:supp_p1}
      \set{\nabla u \neq 0} \subset \set{p(x) = p^-(x)}
      \quad \text{and} \quad \set{b \neq 0} \subset \set{p(x) =  p^+(x)},
    \end{align}
    and
    \begin{align*}
      \int_\Omega \phi(x,\abs{\nabla u})\,dx
      &\lesssim  \int_\Omega
        \abs{x_d}^{p(x_d)-d}          \indicator_{\set{
        2 \abs{x_d} \leq \abs{\bar{x}}
        \leq 4
        \abs{x_d}}} \,dx
      \\
      & \lesssim \int_0^1 \abs{x_d}^{\sigma(\abs{x_d})-1}\,dx_d = \int_0^1
        t^{-1+\sigma(t)}\,dt < \infty.
      \\
      \int_\Omega \phi^*(x,\abs{b})\,dx
      &\lesssim \int_\Omega \abs{x_d}^{(1-d)p'(x)} \indicator_{\set{
        2 \abs{x_d} \leq \abs{\bar{x}} \leq 4
        \abs{x_d}}} \,dx
      \\
      & \lesssim \int_0^1 \abs{x_d}^{(1-d)p'(x) + d-1}\,dx_d
        = \int_0^1 t^{-1+\frac{\sigma(t)}{p_0 + \sigma(t)-1} }
        \,dt < \infty.
    \end{align*}
    This proves $\nabla u \in L^\px(\Omega)$
    and~$b \in L^\pdx(\Omega)$.
  \item 
    We continue with the sub-dimensional case $1< p_0 < d$.  Let
    $\rho \in C^\infty(\setR^d \setminus \frS)$ be such that (using
    Lemma~\ref{lem:smooth-indicator})
    \begin{enumerate}
    \item
      $\indicator_{\set{d(\bar{x},\frC^{d-1}_\lambda) \leq \frac 12 \abs{x_d}}}
      \leq \rho_a \leq \indicator_{\set{d(\bar{x},\frC^{d-1}_\lambda)
          \leq 2 \abs{x_d}}}$.
    \item
      $\abs{\nabla \rho_a} \lesssim \abs{x_d}^{-1}
      \indicator_{\set{\frac 12 \abs{x_d} \leq d(\bar{x},\frC^{d-1}_\lambda) \leq
          2 \abs{x_d}}}$.
    \end{enumerate}
    and define
    \begin{alignat*}{2}
      &p^-(x) &&:= p_0-\sigma(\abs{x_d}),
      \\
      &p^+(x) &&:= p_0+\sigma(\abs{x_d}),
      \\
      &p(x) &&:= p^-(x)\,(1-\rho_a)(x)  + p^+(x)\rho_a(x) .
    \end{alignat*}  
    Then~$p$ is uniformly continuous with modulus of
    continuity~$\sigma$. It follows by Proposition~\ref{pro:est-uAb},
    Lemma~\ref{lem:cantor-estimates},~\eqref{eq:p-cont-aux} and $p_0=d-\frD$
    that
    \begin{align}
      \label{eq:supp_p2}
      \set{\nabla u \neq 0} \subset \set{p = p^-(x)}
      \quad \text{and} \quad \set{b \neq 0} \subset \set{p(x) = p^+(x)},
    \end{align}
    and
    \begin{align*}
      \lefteqn{\int_\Omega \phi(x,\abs{\nabla u})\,dx
      \lesssim  \int_\Omega
      \abs{x_d}^{
      \sigma(x_d)-(d-\frD)}          \indicator_{\set{
      2 \abs{x_d} \leq d(\bar{x}, \frC^{d-1}_\lambda)
      \leq 4
      \abs{x_d}}} \,dx } \qquad &
      \\
                                & \lesssim \int_0^1 \abs{x_d}^{\sigma(\abs{x_d})-1}\,dx_d = \int_0^1
                                  t^{-1+\sigma(t)}\,dt < \infty,
      \\
      \lefteqn{\int_\Omega \phi^*(x,\abs{b})\,dx
      \lesssim \int_\Omega \abs{x_d}^{(\frD+1-d)p'(x)} \indicator_{\set{
      2 \abs{x_d} \leq \abs{\bar{x}} \leq 4
      \abs{x_d}}} \,dx} \qquad &
      \\
                                & \lesssim \int_0^1 \abs{x_d}^{(\frD+1-d)p'(x) + d-1-\frD}\,dx_d
                                  = \int_0^1 t^{-1+\frac{\sigma(t)}{p_0 + \sigma(t)-1} }\,dt < \infty.
    \end{align*}
    This proves $\nabla u \in L^\px(\Omega)$
    and~$b \in L^\pdx(\Omega)$.
  \item 
    Let us turn to the super-dimensional case~$p_0 > d$.  Let
    $\rho \in C^\infty(\setR^d \setminus \frS)$ be such that (using
    Lemma~\ref{lem:smooth-indicator})
    \begin{enumerate}
    \item
      $\indicator_{\set{d(x_d,\frC_\lambda) \leq \frac 1 2 \abs{\bar{x}}}}
      \leq \rho_a \leq \indicator_{\set{d(x_d,\frC_\lambda)
          \leq 2 \abs{\bar{x}}}}$.
    \item
      $\abs{\nabla \rho_a} \lesssim \abs{\bar{x}}^{-1}
      \indicator_{\set{\frac 12 \abs{\bar{x}} \leq
          d(x_d,\frC_\lambda) \leq 2 \abs{\bar{x}}}}
      $
    \end{enumerate}
    and define
    \begin{alignat*}{2}
      &p^-(x)& &:= p_0-\sigma(\abs{\bar{x}}),
      \\
      &p^+(x)& &:= p_0+\sigma(\abs{\bar{x}}),
      \\
      &p(x)& &:= p^-(x)\,\rho_a(x)  + p^+(x)(1-\rho_a)(x).
    \end{alignat*}  
    Then~$p$ is uniformly continuous with modulus of
    continuity~$\sigma$. It follows by Proposition~\ref{pro:est-uAb},
    Lemma~\ref{lem:cantor-estimates},~\eqref{eq:p-cont-aux},
    $p_0=\frac{d-\frD}{1-\frD}$ and $1-\frD = \frac{d-1}{p_0-1}$
    that
    \begin{align*}
      \label{eq:supp_omega3}
      \set{\nabla u \neq 0} \subset \set{p(x) =p^-(x)}
      \quad \text{and} \quad \set{b \neq 0} \subset \set{p(x) =p^+(x) },
    \end{align*}
    and
    \begin{align*}
      \lefteqn{\int_\Omega \phi(x,\abs{\nabla u})\,dx
      \lesssim  \int_\Omega
      \abs{\bar{x}}^{(\frD-1)p(x)}   \indicator_{\set{
      2 \abs{\bar{x}} \leq d(x_d,\frC_\lambda)
      \leq 4
      \abs{\bar{x}}}} \,dx} \qquad&
      \\
                                  & \lesssim \int\limits_{(0,1)^{d-1}} \abs{\bar{x}}^{(1-d)\frac{p(x)-1}{p_0-1}} \,d\bar{x} \leq \int_0^{\sqrt{d}}
                                    t^{\frac{d-1}{p_0-1} \sigma(t) -1}\,dt < \infty,
      \\
      \lefteqn{\int_\Omega \phi^*(x,\abs{b})\,dx
      \lesssim \int_\Omega  \abs{\bar{x}}^{(1-d)p'(x) } \indicator_{\set{
      2 \abs{\bar{x}} \leq d(x_d,\frC_\lambda) \leq 4
      \abs{\bar{x}}}} \,dx} \qquad &
      \\
                                  & \lesssim \int\limits_{(0,1)^{d-1}} \abs{\bar{x}}^{(\frD+1-d)p'(x) + d-1-\frD}\,d\bar{x}
                                    \lesssim \int_0^{\sqrt{d}} t^{-1+\frac{d-1}{(p_0+\sigma(t)-1)(p_0-1)}\sigma(t)} \,dt < \infty.
    \end{align*}
    This proves $\nabla u \in L^\px(\Omega)$
    and~$b \in L^\pdx(\Omega)$.
  \end{enumerate}
  We have proved $\nabla u \in L^\px(\Omega)$
  and~$b \in L^\pdx(\Omega)$ for all~$1 < p_0 < \infty$.  We are now
  going to verify Assumption~\ref{ass:harmonic}. We restrict ourselves
  to the sub-dimensional case, since the other cases are similar.

  We will show that
  \begin{align*}
    \mathcal{F}(tu) + \mathcal{F}^*(sb) \leq \tfrac 12  st
  \end{align*}
  for suitable large~$s,t$.

  Let $s,t > 1$ (to be fixed later).  For $\epsilon>0$ we define
  the~$\epsilon$-neighborhood of~$\frS$ by
  \begin{align*}
    \frS_\epsilon := \set{x \,:\, d(x,\frS) \leq 5 \epsilon}.
  \end{align*}
  Now, with the definition of~$p^\pm$ and~\eqref{eq:supp_p1} we obtain
  \begin{align*}
    \lefteqn{\mathcal{F}(tu ) + \mathcal{F}^*(sb)} \quad
    &
    \\
    &=   \int_{\frS_\epsilon} \phi(x, t\abs{\nabla u})\,dx +  \int_{\Omega \setminus \frS_\epsilon} \phi(x, t\abs{\nabla
      u})\,dx 
    \\
    &+ \int_{\frS_\epsilon} \phi^*(x, s\abs{b})\,dx  + \int_{\Omega
      \setminus \frS_\epsilon} \phi^*(x, s\abs{b})\,dx 
    \\
    &\leq 
      t^{p_0} \int_{\frS_\epsilon} \phi(x, \abs{\nabla u})\,dx +  t^{p_0 - \sigma(\epsilon)} \int_{\Omega \setminus
      \frS_\epsilon} \phi(x, t\abs{\nabla u})\,dx 
    \\
    &\quad 
      +  s^{p_0'} \int_{\frS_\epsilon} \phi^*(x, s\abs{b})\,dx 
      + s^{(p_0+\sigma(\epsilon))'}
      \int_{\Omega \setminus \frS_\epsilon} \phi^*(x, \abs{b})\,dx
    \\
    &\leq 
      \tfrac{1}{p_0}t^{p_0}  \bigg(p_0 \int_{\frS_\epsilon} \phi(x,
      \abs{\nabla u})\,dx +  p_0 t^{- \sigma(\epsilon)} \int_{\Omega \setminus
      \frS_\epsilon} \phi(x, t\abs{\nabla u})\,dx  \bigg)
    \\
    &\quad 
      +  \frac{1}{p_0'} s^{p_0'} \bigg( p_0'\int_{\frS_\epsilon} \phi^*(x, s\abs{b})\,dx 
      + p_0' s^{(p_0+\sigma(\epsilon))'-p_0'}
      \int_{\Omega \setminus \frS_\epsilon} \phi^*(x, \abs{b})\,dx \bigg)
    \\
    &=:       \tfrac{1}{p_0}t^{p_0} \big( \textrm{I} + \textrm{II} \big)
    +       \tfrac{1}{p_0'}s^{p_0'} \big( \textrm{III} + \textrm{IV} \big)
  \end{align*}
  Since
  $\phi(\cdot, \abs{\nabla u}), \phi^*(\cdot, \abs{b}) \in
  L^1(\Omega)$, we can choose~$\epsilon_0>0$ small such
  that~$\textrm{I},\textrm{III} < \tfrac 14$ for
  all~$\epsilon \in (0, \epsilon_0)$. Then choose~$t_0,s_0$ large
  (depending on~$\epsilon_0$) such that
  ~$\textrm{II},\textrm{IV} < \tfrac 14$ for all~$t \geq t_0$ and
  $s\geq s_0$. Thus, we obtain
  \begin{align*}
    \mathcal{F}(tu ) + \mathcal{F}^*(sb) &\leq \tfrac 12 \Big(
                                           \tfrac{1}{p_0}t^{p_0} +
                                           \tfrac{1}{p_0'}s^{p_0'} \Big)
  \end{align*}
  for all~$t\geq t_0$ and $s \geq s_0$. Now, the choice $s:=t^{p_0-1}$
  implies for large~$t$ that
  \begin{align*}
    \mathcal{F}(tu ) + \mathcal{F}^*(sb) &\leq \tfrac 12 st < st.
  \end{align*}
  This proves Assumption~\ref{ass:harmonic}.

  Overall, we have constructed $u$, $b$, and $p$ such that the
  Assumptions~\ref{ass:HneqW}, \ref{ass:gap} and~\ref{ass:harmonic}
  hold.  Now, the claim follows from the results from
  Theorem~\ref{thm:HneqW}, \ref{thm:gap} and~\ref{thm:harmonic} of
  Section~\ref{sec:consequences}.
\end{proof}
Theorem~\ref{thm:main-px} and Theorem~\ref{thm:main-px-cont} show in
particular that the dimensional threshold is not important for the
presence of the Lavrentiev gap and the non-density of smooth
functions.
\begin{remark}
  \label{rem:greek}
  At this point we also have to mention the work of~\cite{KosYan15},
  since it seemingly contradicts our results. The authors claimed that
  $H^{1,\px}(\setR^d)=W^{1,\px}(\setR^d)$ if~$p^- \geq d$
  \cite[Theorem~4.1]{KosYan15} or if $2 \leq p^- < d$ and
  $p^+ < \frac{d p^-}{d-p^-}$ \cite[Theorem~4.2]{KosYan15}. Our
  examples show that both claims are wrong.\footnote{The mistake is in
    Theorem~3.3 of~\cite{KosYan15}. On page~203 the mollified
    gradient~$D_j \omega^i_\lambda$ has no majorant in~$L^{\px}$
    independent of the mollification parameter~$\lambda$.}
\end{remark}

\subsection{Double Phase Potential}
\label{ssec:double-phase-potent}

In this section we study the double phase model. In particular,
we assume that
\begin{align*}
  \phi(x,t) &= \tfrac 1p t^p + a(x) \tfrac
  1q t^q = \tfrac 1p t^p + \tfrac 1q
  (\omega(x)t)^q,
\end{align*}
where $1 < p < q$ with a weights $a,\omega \geq 0$. The corresponding
energy is
\begin{align*}
  \mathcal{F}(w) &= \int_\Omega \tfrac 1p \abs{\nabla w}^p + a(x) \tfrac
                   1q \abs{\nabla w}^q\,dx = \int_\Omega \tfrac 1p \abs{\nabla w}^p + \tfrac 1q
                   \big(\omega(x) \abs{\nabla w}\big)^q\,dx.
\end{align*}
Let us denote by~$C^k$ the space of $k$-times differentiable
functions. Moreover, denote by~$C^{k+\beta}$ for $\beta \in (0,1)$ and
$k \in \setN_0$ the space of functions from~$C^k$ whose $k$-th
derivatives are $\beta$-H\"older continuous.
\begin{theorem}
  \label{thm:main-double-phase}
  Let~$\Omega=(-1,1)^d$ with $d\geq 2$. Let $p>1$ and
  $q > p + \alpha\, \max \set{1,\frac{p-1}{d-1}}$
  with~$\alpha>0$. Then there
  exists~$a \in C^{\alpha}(\overline{\Omega})$ such that
  \begin{enumerate}
  \item $H^{1,\phix}(\Omega) \neq W^{1,\phix}(\Omega)$ and
    $H_0^{1,\phix}(\Omega) \neq W_0^{1,\phix}(\Omega)$.
  \item There exists a linear, continuous
    functional~$\mathcal{S}^\circ\,:\, W^{1,\phix}(\Omega) \to \setR$ such
    that functional~$\mathcal{G} := \mathcal{F} + \mathcal{S}^\circ$
    has a Lavrentiev gap, i.e.
    \begin{align*}
      \inf \mathcal{G}\big(W^{1,\phix}_0(\Omega) \big)
      &<
        \inf \mathcal{G}\big(H^{1,\phix}_0(\Omega)\big) =0.
    \end{align*}
  \item The notions of $\px$-harmonic functions with respect to
    $W^{1,\px}$ and $H^{1,\px}$ differ.
  \end{enumerate}
\end{theorem}
\begin{proof}
  Since $q>p + \alpha\, \max \set{1,\frac{p-1}{d-1}}$  and $p +
  \alpha\, \max \set{1,\frac{p-1}{d-1}}$ is continuous in~$p$, we can
  choose~$p_0>p$ such that $q > p_0 + \alpha\, \max
  \set{1,\frac{p_0-1}{d-1}}$. In particular, $p < p_0 < q$.
  
  Now, let~$u,b$ be as in Definition~\ref{def:fractal-examples} and
  $\mathcal{S}^\circ$ as in~\eqref{eq:defS}.  We begin with the
  definition of our weight~$a(x)$,
  resp.~$\omega(x) = (a(x))^{\frac 1q}$.
  \begin{enumerate}
  \item (Matching the dimension; Zhikov) $p_0=d$:
    Let~$\theta_a \in C^\infty_0((0,\infty))$ be such that
    $\indicator_{(2,\infty)} \leq \theta_a \leq
    \indicator_{(\frac 12,
      \infty)}$ and $\norm{\theta_a'}_\infty \le 6$. We  define
    \begin{align*}
      a(x) &:= \abs{x_d}^\alpha \theta_a\bigg(\frac{\abs{x_d}}{\abs{\bar{x}}}\bigg).
    \end{align*}
  \item (Sub-dimensional) $1<p_0<d$: Let  $\rho \in C^\infty(\setR^d \setminus \frS)$ be such
    that (using Lemma~\ref{lem:smooth-indicator})
    \begin{enumerate}
    \item
      $\indicator_{\set{d(\bar{x},\frC^{d-1}_\lambda) \leq \frac 12 \abs{x_d}}}
      \leq \rho_a \leq \indicator_{\set{d(\bar{x},\frC^{d-1}_\lambda)
          \leq 2 \abs{x_d}}}$.
    \item
      $\abs{\nabla \rho_a} \lesssim \abs{x_d}^{-1}
      \indicator_{\set{\frac 12 \abs{x_d} \leq d(\bar{x},\frC^{d-1}_\lambda) \leq
          2 \abs{x_d}}}$.
    \end{enumerate}
    We define
    \begin{align*}
      a(x) &:= \abs{x_d}^\alpha\, \rho_a(x).
    \end{align*}
  \item (Super-dimensional) $p_0 > d$:  Let $\rho_a \in \setR^d$ be such
    that (using Lemma~\ref{lem:smooth-indicator})
    \begin{enumerate}
    \item
      $\indicator_{\set{d(x_d,\frC_\lambda) \leq \frac 1 2 \abs{\bar{x}}}}
      \leq \rho_a \leq \indicator_{\set{d(x_d,\frC_\lambda)
          \leq 2 \abs{\bar{x}}}}$.
    \item
      $\abs{\nabla \rho_a} \lesssim \abs{\bar{x}}^{-1}
      \indicator_{\set{\frac 12 \abs{\bar{x}} \leq
          d(x_d,\frC_\lambda) \leq 2 \abs{\bar{x}}}}
      $
    \end{enumerate}
    We define
    \begin{align*}
      a(x) &:= \abs{x_d}^\alpha\, (1-\rho_a)(x).
    \end{align*}
  \end{enumerate}
  Then is follows easily by the properties of~$\rho_a$ that
  $a \in C^\alpha(\overline{\Omega})$. Moreover, in the
  sub-dimensional case and the case of matching the dimension we have
  \begin{subequations}
    \label{eq:supp-a}
    \begin{align}
      \set{\nabla u \neq 0} \subset \set{a=0} \quad \text{and} \quad
      \set{b \neq 0} \subset \set{a =\abs{x_d}^\alpha},
    \end{align}
    while in the super-dimensional case we have
    \begin{align}
      \set{\nabla u \neq 0} \subset \set{a=0} \quad \text{and} \quad
    \set{b \neq 0} \subset \set{a =\abs{\bar{x}}^\alpha},
    \end{align}
  \end{subequations}
  Thus, it follows from $p < p_0$ and Corollary~\ref{cor:ub-p0} that
  \begin{align*}
    \norm{\nabla u}_{L^{\phix}(\Omega)}
    &\lesssim
      \norm{\nabla u}_{L^{p}(\Omega)} \leq 
      \norm{\nabla u}_{L^{p_0,\infty}(\Omega)} < \infty.
  \end{align*}
  In the following let~$x$ be such that $b(x)>0$. Then
  $a(x) = \abs{x_d}^\alpha > 0$ for $p_0 \leq d$ and
  $a(x) = \abs{\bar{x}}^\alpha > 0$ for $p_0 > d$. Since
  $\phi(x,t) \geq a(x) \frac 1q t^q = \frac 1q (\omega(x) t)^q$, it
  follows that
  $\phi^*(x,t)\leq \frac 1{q'} (t/\omega(x))^{q'} = \frac 1{q'}
  a(x)^{-\frac{1}{q-1}} t^{p'}=: \psi(x,t)$. This implies
  \begin{subequations}
    \label{eq:double-phi*-psi}
    \begin{alignat}{2}
      \phi^*(x,\abs{b})
      &\leq \frac{1}{q'}  \abs{x_d}^{\frac{-\alpha}{q-1}} \abs{b}^{q'} =
        \psi(x,\abs{b}) &\qquad& \text{for $p_0 \leq d$},
      \\
      \intertext{and}
      \phi^*(x,\abs{b})
      &\leq \frac{1}{q'}  \abs{
        \bar{x}}^{\frac{-\alpha}{q-1}} \abs{b}^{q'} =
        \psi(x,\abs{b})& \qquad &\text{for $p_0 > d$},
    \end{alignat}
  \end{subequations}
  We claim that~$\psi(\cdot,\abs{b}) \in L^{r,\infty}(\Omega)$ for
  some~$r>1$. To prove this we distinguish three cases using
  Proposition~\ref{pro:est-uAb}:
  \begin{enumerate}
  \item (Matching the dimension; Zhikov) $p_0=d$:
    \begin{align*}
      \phi^*(x,\abs{b})
      &\lesssim
        \abs{x_d}^{\frac{-\alpha}{q-1}}
        \abs{x_d}^{(1-d)q'} \indicator_{\set{2\abs{\bar{x}} \leq
        \abs{x_d}}}.
    \end{align*}
    Thus, by Lemma~\ref{lem:cantor-weak} we have
    $\phi^*(\cdot,\abs{b}) \in L^{r,\infty}(\Omega)$ with
    $r:= \frac{d(q-1)}{\alpha+(d-1)q}$. Now, $r>1$ is equivalent to
    $q > d+\alpha = p_0 + \alpha$, which is true by choice of~$p_0$.
  \item (Sub-dimensional) $1<p_0=d - \frD <d$: 
    \begin{align*}
      \phi^*(x,\abs{b})
      &\lesssim
        \abs{x_d}^{\frac{-\alpha}{q-1}}
        \abs{x_d}^{(\frD+1-d)q'} \indicator_{\set{d(\bar{x}, \frC^{d-1}_\lambda) \leq
        \frac 12\abs{x_d}}}.
    \end{align*}
    Thus, by Lemma~\ref{lem:cantor-weak} we have
    $\phi^*(\cdot,\abs{b}) \in L^{r,\infty}(\Omega)$ with
    $r:= \frac{(d-\frD)(q-1)}{\alpha +(d-\frD-1)q}$. Now, $r>1$ is
    equivalent to~$q > d-\frD + \alpha = p_0 + \alpha$, which is true
    by choice of~$p_0$.
  \item (Super-dimensional) $p_0=
    \frac{d-\frD}{1-\frD} > d$. 
    \begin{align*}
      \phi^*(x,\abs{b})
      &\lesssim
        \tfrac 1{q'} \abs{
        \bar{x}}^{\frac{-\alpha}{q-1}}
        \abs{\bar{x}}^{(1-d)q'} \indicator_{\set{d(x_d,\frC_\lambda) \leq
        4\abs{\bar{x}}}}.
    \end{align*}
    Thus, by Lemma~\ref{lem:cantor-weak} (applied to~$m=1$) we have
    $\phi^*(\cdot,\abs{b}) \in L^{r,\infty}(\Omega)$ with
    $r:= \frac{(d-\frD)(q-1)}{\alpha+(d-1)q}$. Now, $r>1$ is
    equivalent to
    $q > \frac{\alpha+d-\frD}{1-\frD} = \frac{\alpha}{1-\frD} +
    p_0$. Since $p_0 =\frac{d-\frD}{1-\frD}$, we have
    $1-\frD = \frac{d-1}{p_0-1}$. Thus, $r>1$ is equivalent to
    $q> p_0 + \alpha \frac{d-1}{p_0-1}$, which is true by choice
    of~$p_0$.
  \end{enumerate}
  We have proved that $\psi(\cdot,\abs{b}) \in L^{r,\infty}(\Omega)$
  some~$r>1$, which proves $\psi(\cdot,\abs{b}) \in L^1(\Omega)$. Due
  to~\eqref{eq:double-phi*-psi} this implies
  $\phi^*(\cdot,\abs{b}) \in L^1(\Omega)$ and
  therefore~$b \in L^\phidx(\Omega)$. Overall, we have verified
  Assumption~\ref{ass:HneqW}.
  
  Since $p,q \in (1,\infty)$, it follows that~$\phi$ and $\phi^*$ satisfy
  the~$\Delta_2$ condition. Thus Assumption~\ref{ass:gap} also holds.

  Let $s,t>0$ (to be fixed later).  By~\eqref{eq:supp-a} we have
  \begin{align*}
    \mathcal{F}(tu) &= t^p \mathcal{F}(u).
  \end{align*}
  Moreover, by~\eqref{eq:double-phi*-psi} we have
  \begin{align*}
    \mathcal{F}^*(sb)
    &\leq  \int_\Omega \psi(x,s\abs{b})\,dx
      = s^{q'}  \int_\Omega \psi(x,\abs{b})\,dx.
  \end{align*}
  Thus, we have
  \begin{align*}
    \mathcal{F}(tu) +
    \mathcal{F}^*(sb)
    &\leq t^p \mathcal{F}(u) + s^{q'}  \int_\Omega \psi(x,\abs{b})\,dx.
  \end{align*}
  Since~$p<p_0 < q$, we get
  \begin{align*}
    \mathcal{F}(tu) +
    \mathcal{F}^*(sb)
    &\leq \frac 1{p_0} t^{p_0} \bigg( p_0 t^{p-p_0}
      \mathcal{F}(u)\bigg)
      + \frac{1}{p_0} s^{p_0'} \bigg( p_0' s^{q'-p_0'}
      \int_\Omega \psi(x,\abs{b})\,dx \bigg).
  \end{align*}
  Now, fix $s:= t^{p_0-1}$. Then for suitable large~$t$ (and
  therefore large~$s$) we obtain
  \begin{align*}
    \mathcal{F}(tu) +
    \mathcal{F}^*(sb)
    &\leq \tfrac{1}{p_0} t^{p_0} \cdot \tfrac 12 +
                \tfrac{1}{p_0'} s^{p_0'} \tfrac 12 = \tfrac 12 ts < ts,
  \end{align*}
  where we have used~$p<p_0$ and $q' < p_0'$.  This proves
  Assumption~\ref{ass:harmonic}.

  Overall, we have constructed $u$, $b$, and $p$ such that the
  Assumptions~\ref{ass:HneqW}, \ref{ass:gap} and~\ref{ass:harmonic}
  holds.  Now, the claim follows from the results from
  Theorem~\ref{thm:HneqW}, \ref{thm:gap} and~\ref{thm:harmonic} of
  Section~\ref{sec:consequences}.
\end{proof}

\begin{remark}
  Theorem~\ref{thm:main-double-phase} shows that the dimensional
  threshold is not important for the presence of the Lavrentiev gap
  and the non-density of smooth functions. (Recall that the previous
  examples needed $p<d< d+\alpha < p$ crossing the dimension.)

  Since we have overcome the dimensional threshold, it might be
  surprising that we obtain different conditions on~$p$ and $q$
  for~$p\leq d$ and $p>d$. Therefore, let us explain in the following
  that our conditions are sharp:

  Consider first the case~$p \leq d$. In this case we get Lavrentiev
  gap for~$q> p+\alpha$. Now, it has been shown in~\cite{ColMin15}
  that if~$h$ is a bounded minimizer of~$\mathcal{F}$ and
  $q \leq p+\alpha$, then~$h$ is automatically in~$W^{1,q}(\Omega)$.
  Since
  $W^{1,q}(\Omega)=H^{1,q}(\Omega) \embedding H^{1,\phix}(\Omega)$ it
  follows that~$h \in H^{1,\phix}(\Omega)$ and there is no Lavrentiev
  gap. This shows that our condition~$q>p+\alpha$ is sharp.  The
  boundeness of the minimizer is a reasonable assumption due to the
  maximum principle.  Also note that our function~$u$ is also
  bounded. This is reflected by the fact that functions
  in~$W^{1,\phix}(\Omega)$ can always be approximated by
  $L^\infty(\Omega) \cap W^{1,\phix}(\Omega)$ functions by means of
  truncation.

  Now, consider the case~$p > d$. In this case it has been shown in
  \cite[Theorem~1.4]{BarColMin18} that if~$h$ is a minimizer
  of~$\mathcal{F}$,~$h \in C^{0,\gamma}(\overline{\Omega})$ and
  $q \leq p +\frac{\alpha}{1-\gamma}$, then~$h$ is automatically
  in~$W^{1,q}(\Omega)$. Again
  $W^{1,q}(\Omega)=H^{1,q}(\Omega) \embedding H^{1,\phix}(\Omega)$
  implies that~$h \in H^{1,\phix}(\Omega)$ and there is no Lavrentiev
  gap. Now, in our example we constructed a
  function~$u \in \mathcal{C}^{0,\frD}(\overline{\Omega})$ with
  $\frD = \frac{p_0-d}{p_0-1}$. Thus, we can compare our condition
  $q > p + \alpha \frac{p-1}{d-1}$ for the Lavrentiev gap with
  $q \leq p+\frac{\alpha}{1-\gamma}$ for~$\gamma := \frD$ for the
  absence of the Lavrentiev gap. Now,
  $p+ \frac{\alpha}{1-\frD} = p + \alpha\frac{p_0-1}{d-1}$.  Since
  $p_0$ can be chosen close to~$p$ this shows, that our
  condition~$q > p + \alpha\frac{p-1}{d-1}$ is sharp.
\end{remark}

\begin{remark}
  Fonseca, Mal{\'y} and Mingione studied in~\cite{FonMalMin04} the
  size of possible singular sets of minimizer of the double phase
  potential. For $1 < p < d < d+\alpha < q < \infty$ they constructed
  a weight~$a$ such that the singular set has Hausdorff dimension
  larger than $d-p-\epsilon$.

  Let us compare this to our result. Since $p<d$, we can choose
  $p_0 = p+\delta$ with~$\delta>0$ small. Thus, our function~$u$ has a
  singular set of Hausdorff
  dimension~$\mathcal{D}= d-p_0= d-p-\delta$. In particular, we obtain
  a singular set of the same size. Note however, that our function~$u$
  is not a minimizer yet, but we expect that we can use~$u$ as a
  competitor to find a minimizer with a singular set of same Hausdorff
  dimension. We will work on this question in a future project. Using
  out method we hope to overcome the dimensional threshold $p < d <
  d+\alpha < q$ from~\cite{FonMalMin04}.
\end{remark}



  

\subsection{\texorpdfstring{Weighted $p$-Energy}{Weighted p-Energy}}
\label{ssec:weighted-p-energy}

In this section we study the model with weighted $p$-energy. In
particular, we assume that
\begin{align*}
  \phi(x,t) &= \tfrac 1p a(x) t^p =\tfrac 1p (\omega(x) t)^p
\end{align*}
where $1 < p < \infty$ and weights $a, \omega \geq 0$ (almost
everywhere). The corresponding energy is
\begin{align*}
  \mathcal{F}(w) &= \int_\Omega \tfrac 1p a(x) \abs{\nabla w}^p \,dx = \int_\Omega
                   \tfrac 1p (\omega(x) \abs{\nabla w})^p\,dx. 
\end{align*}
\begin{definition}
  The weight~$a(x)$ belongs to the Muckenhoupt class~$A_p$ if
  \begin{align*}
    \sup \left(\frac {1} {\abs{Q}} \int_Q a(x)\, dx \right)\left(\frac {1} {\abs{Q}} \int_Q (a(x))^{-\frac{1}{p-1}}\, dx \right)^{p-1}\le K,
  \end{align*}
  where the supremum is taken over all cubes~$Q$.
\end{definition}

Our main result of the weighted $p$-energy is the following:
\begin{theorem}
  \label{thm:main-weighted-p}
  Let~$\Omega=(-1,1)^d$ with $d\geq 2$ and $1< p < \infty$.  Then
  there exists weights $a^-$ and $a^+$ of Muckenhoupt class~$A_p$ and
  another weight~$a$ with $a^+ \leq a \leq a^+$ such that the
  following holds: 
  \begin{enumerate}
  \item \label{itm:main-weighted-p1} $H^{1,\phix}(\Omega) \neq W^{1,\phix}(\Omega)$ and
    $H_0^{1,\phix}(\Omega) \neq W_0^{1,\phix}(\Omega)$.
  \item \label{itm:main-weighted-p2} There exists a linear, continuous
    functional~$\mathcal{S}^\circ\,:\, W^{1,\phix}(\Omega) \to \setR$ such
    that~$\mathcal{G} := \mathcal{F} + \mathcal{S}^\circ$
    has a Lavrentiev gap, i.e.
    \begin{align*}
      \inf \mathcal{G}\big(W^{1,\phix}_0(\Omega) \big)
      &<
        \inf \mathcal{G}\big((H^{1,\phix}_0(\Omega)\big) =0.
    \end{align*}
  \item \label{itm:main-weighted-p3} The notions of $\phix$-harmonic
    functions with respect to $W^{1,\phix}$ and $H^{1,\phix}$ differ.
  \end{enumerate}
  It is possible to choose either~$a$ or $\frac{1}{a}$ bounded.
\end{theorem}
\begin{proof}
  Let $\alpha,\beta, \gamma$ be such that
  \begin{align*}
    -\frac{1}{p} < \alpha < \gamma < \beta < \frac{1}{p'}. 
  \end{align*}
  If $\alpha,\beta \geq 0$, then $a$ will be bounded. If
  $\alpha, \beta\leq 0$, then $\frac{1}{a}$ will be bounded
  
  If $\gamma = 1-\frac{d}{p}$, let $p_0 := d$. If
  $\gamma \in (1-\frac{d}{p},\frac 1{p'})$, choose $p_0 \in (1,d)$
  (and hence~$p_0 = d-\frD$) such that $\gamma = 1-\frac{p_0}{p}$. If
  $\gamma \in (-\frac{d-1}{p},1 -\frac{d}{p})$ choose~$p_0>d$ (and
  hence~$p_0 = \frac{d-\frD}{1-\frD}$) such that
  $\gamma = \frac{d-1}{p_0-1}(1-\frac{p_0}{p})= \frac{d-1}{p}
  \frac{p-p_0}{p_0-1}$. Moreover, let~$\epsilon>0$ be another
  parameter. To obtain~\ref{itm:main-weighted-p3}, we have to
  choose~$\epsilon>0$ later small.

  Now, let~$u,b$ be as in Definition~\ref{def:fractal-examples} and
  $\mathcal{S}^\circ$ as in~\eqref{eq:defS}. For our construction and
  proof we distinguish the three cases $p_0=d$, $p_0 < d$ and $p_0>d$.

  \begin{enumerate}
  \item 
    We begin with the case of matching the dimension~$p_0=d$.
    Let~$\theta_a \in C^\infty_0((0,\infty))$ be such that
    $\indicator_{(2,\infty)} \leq \theta_a \leq \indicator_{(\frac 12,
      \infty)}$ and $\norm{\theta_a'}_\infty \le 6$. We  define
    \begin{align*}
      \omega^-(x) &:= \epsilon\,\abs{x_d}^\beta,
      \\
      \omega^+(x) &:= \abs{x_d}^\alpha,
      \\
      \omega(x) &:= \omega^-(x)
                  \theta_a\bigg(\frac{\abs{x_d}}{\abs{\bar{x}}}\bigg) +
                  \omega^+(x) \bigg( 1 -
                  \theta_a\bigg(\frac{\abs{x_d}}{\abs{\bar{x}}}\bigg) \bigg).
    \end{align*}
    We can assume that~$\epsilon \leq d^{\frac{\beta-\alpha}{2}}$, so
    that $\omega^- \leq \omega \leq \omega^+$. The weights~$\omega^-$
    and $\omega^+$ are of Muckenhoupt class~$A_p$, since
    $\alpha, \beta \in (-\frac 1p, \frac{1}{p'})$.  It follows by
    Proposition~\ref{pro:est-uAb} that
    \begin{align}
      \label{eq:supp_omega1}
      \set{\nabla u \neq 0} \subset \set{\omega= \epsilon \abs{x_d}^\beta}
      \quad \text{and} \quad \set{b \neq 0} \subset \set{\omega
      =\abs{x_d}^\alpha},
    \end{align}
    and
    \begin{align*}
      \phi(x,\abs{\nabla u}) &\lesssim  \epsilon^p\, \big( \abs{x_d}^\beta
                               \abs{x_d}^{-1} \big)^p
                               \indicator_{\set{
                               2 \abs{x_d} \leq \abs{\bar{x}} \leq 4
                               \abs{x_d}}},
      \\
      \phi^*(x,\abs{b}) &\lesssim \big( \abs{x_d}^{-\alpha}
                          \abs{x_d}^{1-d} \big)^{p'} \indicator_{\set{
                          2 \abs{x_d} \leq \abs{\bar{x}} \leq 4 \abs{x_d}}}.
    \end{align*}
    Thus, by Lemma~\ref{lem:cantor-weak} we have
    $\phi(\cdot,\abs{\nabla u}) \in L^{r,\infty}(\Omega)$ with
    $r=\frac{d}{(1-\beta)p}>1$ using $\beta > \gamma=1-\frac{d}{p}$.
    This proves~$\nabla u \in L^\phix(\Omega)$.  Moreover, by
    Lemma~\ref{lem:cantor-weak} we have
    $\phi^*(\cdot,\abs{b}) \in L^{s,\infty}(\Omega)$ with
    $s=\frac{d}{(\alpha+d-1)p'}>1$ using
    $\alpha < \gamma=1-\frac{d}{p}$.  This
    proves~$b \in L^\phidx(\Omega)$.
    
  \item 
    We continue with the sub-dimensional case $1< p_0 < d$.  Let
    $\rho_a \in C^\infty(\setR^d \setminus \frS)$ be such that (using
    Lemma~\ref{lem:smooth-indicator})
    \begin{enumerate}
    \item
      $\indicator_{\set{d(\bar{x},\frC^{d-1}_\lambda) \leq \frac 12 \abs{x_d}}}
      \leq \rho_a \leq \indicator_{\set{d(\bar{x},\frC^{d-1}_\lambda)
          \leq 2 \abs{x_d}}}$.
    \item
      $\abs{\nabla \rho_a} \lesssim \abs{x_d}^{-1}
      \indicator_{\set{\frac 12 \abs{x_d} \leq d(\bar{x},\frC^{d-1}_\lambda) \leq
          2 \abs{x_d}}}$.
    \end{enumerate}
    and define
    \begin{align*}
      \omega^-(x) &:= \epsilon\,\abs{x_d}^\beta,
      \\
      \omega^+(x) &:= \abs{x_d}^\alpha,
      \\
      \omega(x) &:= \omega^-(x)\, \rho_a(x) + \omega^+(x)(1-\rho_a)(x).
    \end{align*}  
    We can assume that~$\epsilon \leq d^{\frac{\beta-\alpha}{2}}$, so
    that $\omega^- \leq \omega \leq \omega^+$. The weights~$\omega^-$
    and $\omega^+$ are of Muckenhoupt class~$A_p$, since
    $\alpha, \beta \in (-\frac 1p, \frac{1}{p'})$.  It follows by
    Proposition~\ref{pro:est-uAb} that
    \begin{align}
      \label{eq:supp_omega2}
      \set{\nabla u \neq 0} \subset \set{\omega=\epsilon \abs{x_d}^\beta}
      \quad \text{and} \quad \set{b \neq 0} \subset \set{\omega
      =\abs{x_d}^\alpha},
    \end{align}
    and
    \begin{align*}
      \phi(x,\abs{\nabla u})
      &\lesssim \epsilon^p\,(\abs{x_d}^\beta \abs{x_d}^{-1})^p \indicator_{\set{2
        \abs{x_d} \leq d(\bar{x}, \frC^{d-1}_\lambda) \leq 4
        \abs{x_d}}},
              \\
            \phi^*(x,\abs{b})
      &\lesssim (\abs{x_d}^{-\alpha} \abs{x_d}^{\frD + 1-d})^{p'}
        \indicator_{\set{d(\bar{x}, \frC^{d-1}_\lambda) \leq \frac 12 \abs{x_d}}}.
    \end{align*}
    Thus, by Lemma~\ref{lem:cantor-weak} we have
    $\phi(\cdot,\abs{\nabla u}) \in L^{r,\infty}(\Omega)$ with
    $r=\frac{d-\frD}{(1-\beta)p}>1$ using
    $\beta > \gamma = 1 - \frac{d-\frD}p = 1 - \frac{p_0}{p}$. This
    proves~$\nabla u \in L^\phix(\Omega)$. Moreover, by
    Lemma~\ref{lem:cantor-weak} we have
    $\phi^*(\cdot,\abs{b}) \in L^{s,\infty}(\Omega)$ with
    $s=\frac{d-\frD}{d-1-\frD+\alpha}>1$ using
    $\alpha < \gamma = 1 - \frac{d-\frD}p = 1 - \frac{p_0}{p}$. This
    proves~$b \in L^\phidx(\Omega)$.
  \item 
    Let us turn to the super-dimensional case~$p_0 > d$.  Let
    $\rho_a \in C^\infty(\setR^d \setminus \frS)$ be such that (using
    Lemma~\ref{lem:smooth-indicator})
    \begin{enumerate}
    \item
      $\indicator_{\set{d(x_d,\frC_\lambda) \leq \frac 1 2 \abs{\bar{x}}}}
      \leq \rho_a \leq \indicator_{\set{d(x_d,\frC_\lambda)
          \leq 2 \abs{\bar{x}}}}$.
    \item
      $\abs{\nabla \rho_a} \lesssim \abs{\bar{x}}^{-1}
      \indicator_{\set{\frac 12 \abs{\bar{x}} \leq
          d(x_d,\frC_\lambda) \leq 2 \abs{\bar{x}}}}
      $
    \end{enumerate}
    and define
    \begin{align*}
      \omega^-(\bar{x}) &:= \epsilon\,\abs{\bar{x}}^\beta,
      \\
      \omega^+(x) &:= \abs{\bar{x}}^\alpha,
      \\
      \omega(x) &:= \omega^-(x)\, (1-\rho_a)(x) + \omega^+(x) \rho_a(x).
    \end{align*}
    The weights~$\omega^-$ and $\omega^+$ are of Muckenhoupt
    class~$A_p$, since $\alpha, \beta \in (-\frac{d-1}p, \frac{d-1}{p'})$.
    It follows by Proposition~\ref{pro:est-uAb} that
    \begin{align*}
      \label{eq:supp_omega3}
      \set{\nabla u \neq 0} \subset \set{\omega=\epsilon \abs{\bar{x}}^\beta}
      \quad \text{and} \quad \set{b \neq 0} \subset \set{\omega
      =\abs{\bar{x}}^\alpha},
    \end{align*}
    and
    \begin{align*}
      \phi(x,\abs{\nabla u})
      &\lesssim \epsilon^p (\abs{\bar{x}}^\beta \abs{\bar{x}}^{\frD-1})^p \indicator_{\set{2
        \abs{\bar{x}} \leq d(x_d, \frC_\lambda) \leq 4
        \abs{\bar{x}}}},
      \\
      \phi^*(x,\abs{b})
      &\lesssim (\abs{\bar{x}}^{-\alpha} \abs{\bar{x}}^{1-d})^{p'}
        \indicator_{\set{d(x_d, \frC^{d-1}_\lambda) \leq \frac 12 \abs{\bar{x}}}}.
    \end{align*}
    Thus, by Lemma~\ref{lem:cantor-weak} we have
    $\phi(\cdot,\abs{\nabla u}) \in L^{r,\infty}(\Omega)$ with
    $r=\frac{d-\frD}{(1-\frD-\alpha)p}>1$ using
    $\beta > \gamma = 1-\frD -\frac{d-\frD}{p} =
    \frac{d-1}{p_0-1}(1-\frac{p_0}{p})$. This
    proves~$\nabla u \in L^\phix(\Omega)$. Moreover, by
    Lemma~\ref{lem:cantor-weak} we have
    $\phi^*(\cdot,\abs{b}) \in L^{s,\infty}(\Omega)$ with
    $s=\frac{d-\frD}{(d-1+\beta)p'}>1$ using
    $\beta < \gamma = 1-\frD -\frac{d-\frD}{p} =
    \frac{d-1}{p_0-1}(1-\frac{p_0}{p})$. This
    proves~$b \in L^\phidx(\Omega)$.
  \end{enumerate}
  Now $1<p<\infty$ ensures that $\phi$ and $\phi^*$ satisfy the
  $\Delta_2$-condition.

  Overall, we have constructed $u$, $b$, and $p$ such that the
  Assumptions~\ref{ass:HneqW} and \ref{ass:gap} hold.

  We are now going to verify Assumption~\ref{ass:harmonic}. This will
  be the step, where we have to choose~$\epsilon>0$ small enough. It
  suffices to prove Assumption~\ref{ass:harmonic} in the case
  of~$p_0 < d$. (The other case are completely analogous.)
  
  Let $s,t>0$ (to be fixed later).  By~\eqref{eq:supp_omega2} we have
  \begin{align*}
    \mathcal{F}(tu) &\lesssim (\epsilon t)^p \int_\Omega
                      (\abs{x_d}^\beta \abs{x_d}^{-1})^p \indicator_{\set{2
                      \abs{x_d} \leq d(\bar{x}, \frC^{d-1}_\lambda) \leq 4
                      \abs{x_d}}}
                      \,dx.
  \end{align*}
  Moreover, by~\eqref{eq:double-phi*-psi} we have
  \begin{align*}
    \mathcal{F}^*(sb)
    &\lesssim s^{p'} \int_\Omega 
      (\abs{x_d}^{-\alpha} \abs{x_d}^{\frD + 1-d})^{p'}
      \indicator_{\set{d(\bar{x}, \frC^{d-1}_\lambda) \leq \frac 12
      \abs{x_d}}}\,dx.
  \end{align*}
  Thus, there exist~$c_1,c_2>0$ such that
  \begin{align*}
    \mathcal{F}(tu) +
    \mathcal{F}^*(sb)
    &\leq c_1 (\epsilon t)^p + c_2 s^{p'}.
  \end{align*}
  Now, the choice $t:=1$, $s:= \epsilon^{p-1}$ and
  $\epsilon = (\frac{1}{2 (c_1 +c_2)})^{\frac 1p}$ implies
  \begin{align*}
    \mathcal{F}(tu) +
    \mathcal{F}^*(sb)
    &\leq \epsilon^p\, (c_1+c_2) = \tfrac 12 ts < ts.
  \end{align*}
  This proves Assumption~\ref{ass:harmonic}.

  Overall, we have constructed $u$, $b$, and $p$ such that the
  Assumptions~\ref{ass:HneqW}, \ref{ass:gap} and~\ref{ass:harmonic}
  holds.  Now, the claim follows from the results from
  Theorem~\ref{thm:HneqW}, \ref{thm:gap} and~\ref{thm:harmonic} of
  Section~\ref{sec:consequences}.
\end{proof}

\begin{remark}
  \label{rem:zhikov-barrier}
  Note, that Zhikov has also provided in~\cite[Section 5.3]{Zhi98}
  another example for~$H^{1,\phix}(\Omega) \neq W^{1,\phi}(\Omega)$
  for the model of the weighted $p$-energy for~$p=2$. He was
  interested in an example, where the weight is bounded.  Since, the
  checker board example does not work in this case, he constructed an
  example using fractals. For this he split the domain~$\Omega$ into
  two parts~$\Omega_1$ and~$\Omega_0$ separated them by another
  part~$\mathcal{N}$ consisting of a Cantor-necklace, see
  Figure~\ref{fig:zhikov-barrier}. He then constructed a 
  weight~$a\in L^{\infty}(\Omega)$ with $\frac 1a \in L^1(\Omega)$ and a
  function~$u \in (L^\infty(\Omega) \cap W^{1,\phix}(\Omega))
  \setminus H^{1,\phix}(\Omega)$ with $u = 0$ on~$\Omega_0$ and~$u=1$
  on~$\Omega_1$. The weight was chosen such that it is integrable on
  each block of the necklace and scaled by the size on the smaller
  blocks. So the weight becomes smaller on smaller blocks. His
  example is contained in the class of possible weights from
  Theorem~\ref{thm:main-weighted-p}.
\end{remark}

\newcommand{\zhikovbarrier}{%
        \fill[white] (-1/3,0) -- (0,1/3) -- (1/3,0) -- (0,-1/3) -- cycle;
      \draw (-1/3,0) -- (0,1/3) -- (1/3,0) -- (0,-1/3) -- cycle;
      \fill[opacity=0.15] (-1/3,0) -- (0,1/3) -- (1/3,0) -- (0,-1/3)
      -- cycle;
}%

\newcommand{\barrierpattern}{%
  \zhikovbarrier

  \begin{scope}[shift={(-2/3,0)},scale=1/3]
    \zhikovbarrier
  \end{scope}
  \begin{scope}[shift={(2/3,0)},scale=1/3]
    \zhikovbarrier
  \end{scope}
  \begin{scope}[shift={(-2/3-2/9,0)},scale=1/9]
    \zhikovbarrier
  \end{scope}
  \begin{scope}[shift={(2/3+2/9,0)},scale=1/9]
    \zhikovbarrier
  \end{scope}
  \begin{scope}[shift={(-2/3+2/9,0)},scale=1/9]
    \zhikovbarrier
  \end{scope}
  \begin{scope}[shift={(2/3-2/9,0)},scale=1/9]
    \zhikovbarrier
  \end{scope}
}

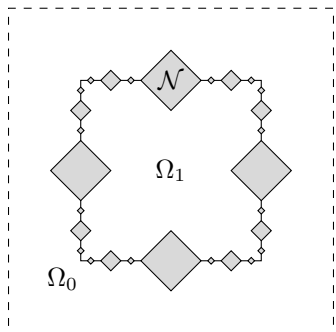
\begin{figure}[!ht]
  \centering

  \begin{tikzpicture}[scale=1.2]
    \draw[dashed] (-1.8,-1.8) rectangle (1.8,1.8);
    \draw (-1,-1) rectangle (1,1);
    \begin{scope}[shift={(0,1)}]
      \barrierpattern
    \end{scope}
    \begin{scope}[shift={(0,-1)}]
      \barrierpattern
    \end{scope}
    \begin{scope}[shift={(1,0)},rotate=90]
      \barrierpattern
    \end{scope}
    \begin{scope}[shift={(-1,0)},rotate=90]
      \barrierpattern
    \end{scope}
    \node at (0,0) {$\Omega_1$};
    \node at (-1.2,-1.2) {$\Omega_0$};
    \node at (0,1) {$\mathcal{N}$};
  \end{tikzpicture}
  
  \caption{Zhikov's fractal barrier}
  \label{fig:zhikov-barrier}
\end{figure}


\begin{thebibliography}{10}

\bibitem{BarColMin18}
Paolo Baroni, Maria Colombo, and Giuseppe Mingione.
\newblock Regularity for general functionals with double phase.
\newblock {\em Calc. Var. Partial Differential Equations}, 57(2):Art. 62, 48,
  2018.

\bibitem{ColMin15}
Maria Colombo and Giuseppe Mingione.
\newblock Bounded minimisers of double phase variational integrals.
\newblock {\em Arch. Ration. Mech. Anal.}, 218(1):219--273, 2015.

\bibitem{CruFio13}
David~V. Cruz-Uribe and Alberto Fiorenza.
\newblock {\em Variable {L}ebesgue spaces}.
\newblock Applied and Numerical Harmonic Analysis. Birkh\"{a}user/Springer,
  Heidelberg, 2013.
\newblock Foundations and harmonic analysis.

\bibitem{Die04}
L.~Diening.
\newblock Maximal function on generalized {L}ebesgue spaces {$L\sp
  {p(\cdot)}$}.
\newblock {\em Math. Inequal. Appl.}, 7(2):245--253, 2004.

\bibitem{DieHHR11}
L.~Diening, P.~Harjulehto, P.~H{\"a}st{\"o}, and M.~R{\r u}{\v z}i{\v c}ka.
\newblock {\em Lebesgue and Sobolev Spaces with Variable Exponents}, volume
  2017 of {\em Lecture Notes in Mathematics}.
\newblock Springer, 1st edition, 2011.

\bibitem{DieHasNek04}
Lars Diening, Peter Hästo, and Ales Nekvinda.
\newblock Open problems in variable exponent lebesgue and sobolev spaces.
\newblock {\em FSDONA04 Proceedings}, pages 38--52, 01 2004.

\bibitem{EdmRak92}
David~E. Edmunds and Ji\v{r}\'{i} R\'{a}kosn\'{i}k.
\newblock Density of smooth functions in {$W^{k,p(x)}(\Omega)$}.
\newblock {\em Proc. Roy. Soc. London Ser. A}, 437(1899):229--236, 1992.

\bibitem{EspLeoMin04}
Luca Esposito, Francesco Leonetti, and Giuseppe Mingione.
\newblock Sharp regularity for functionals with {$(p,q)$} growth.
\newblock {\em J. Differential Equations}, 204(1):5--55, 2004.

\bibitem{FonMalMin04}
Irene Fonseca, Jan Mal\'{y}, and Giuseppe Mingione.
\newblock Scalar minimizers with fractal singular sets.
\newblock {\em Arch. Ration. Mech. Anal.}, 172(2):295--307, 2004.

\bibitem{HarHas19}
Petteri Harjulehto and Peter H{\"a}st{\"o}.
\newblock {\em Generalized Orlicz Spaces}.
\newblock Springer International Publishing, Cham, 2019.

\bibitem{Has05}
Peter~A. H\"{a}st\"{o}.
\newblock Counter-examples of regularity in variable exponent {S}obolev spaces.
\newblock In {\em The {$p$}-harmonic equation and recent advances in analysis},
  volume 370 of {\em Contemp. Math.}, pages 133--143. Amer. Math. Soc.,
  Providence, RI, 2005.

\bibitem{KokMesRafSam16}
Vakhtang Kokilashvili, Alexander Meskhi, Humberto Rafeiro, and Stefan Samko.
\newblock {\em Integral operators in non-standard function spaces. {V}ol. 1},
  volume 248 of {\em Operator Theory: Advances and Applications}.
\newblock Birkh\"{a}user/Springer, [Cham], 2016.
\newblock Variable exponent Lebesgue and amalgam spaces.

\bibitem{KosYan15}
Thanasis Kostopoulos and Nikos Yannakakis.
\newblock Density of smooth functions in variable exponent {S}obolev spaces.
\newblock {\em Nonlinear Anal.}, 127:196--205, 2015.

\bibitem{Lav26}
M.~Lavrentiev.
\newblock Sur quelques problemes du calcul des variations.
\newblock {\em Ann. Math. Pura Appl.}, 4:107--124, 1926.

\bibitem{Man34}
B.~Mani{\'a}.
\newblock Sopra un esempio di lavrentieff.
\newblock {\em Bull. Un. Mat. Ital.}, pages 147--153, 1934.

\bibitem{Mar89}
Paolo Marcellini.
\newblock Regularity of minimizers of integrals of the calculus of variations
  with nonstandard growth conditions.
\newblock {\em Arch. Rational Mech. Anal.}, 105(3):267--284, 1989.

\bibitem{Sur14}
M.~D. Surnachev.
\newblock Density of smooth functions in weighted sobolev spaces with variable
  exponent.
\newblock {\em Doklady Mathematics}, 89(2):146--150, Mar 2014.

\bibitem{ZhiSur16}
V~V.~Zhikov and Mikhail Surnachev.
\newblock On density of smooth functions in weighted sobolev spaces with
  variable exponents.
\newblock {\em St Petersburg Mathematical Journal}, 27:415--436, 06 2016.

\bibitem{Zhi86}
V.~V. Zhikov.
\newblock Averaging of functionals of the calculus of variations and elasticity
  theory.
\newblock {\em Izv. Akad. Nauk SSSR Ser. Mat.}, 50(4):675--710, 877, 1986.

\bibitem{Zhi98}
V.~V. Zhikov.
\newblock On weighted {S}obolev spaces.
\newblock {\em Mat. Sb.}, 189(8):27--58, 1998.

\bibitem{Zhi04}
V.~V. Zhikov.
\newblock On the density of smooth functions in {S}obolev-{O}rlicz spaces.
\newblock {\em Zap. Nauchn. Sem. S.-Peterburg. Otdel. Mat. Inst. Steklov.
  (POMI)}, 310(Kraev. Zadachi Mat. Fiz. i Smezh. Vopr. Teor. Funkts. 35
  [34]):67--81, 226, 2004.

\bibitem{Zhi95}
Vasili\u{\i}~V. Zhikov.
\newblock On {L}avrentiev's phenomenon.
\newblock {\em Russian J. Math. Phys.}, 3(2):249--269, 1995.

\end{thebibliography}
\end{document}